\documentclass{article}
\usepackage{amsmath,mathrsfs,amssymb,amsthm,amscd}
\usepackage[]{fontenc}
\usepackage[all]{xy}
\usepackage{hyperref}
\usepackage{graphics}
\usepackage{a4wide}
\usepackage{setspace}
\usepackage{palatino}
\usepackage{color}

\newcounter{EQNR}
\renewcommand{\contentsname}{Contents\\
{\footnotesize\normalfont(A table of contents should normally not be included)}}

\theoremstyle{plain}
\newtheorem{thm}{Theorem}[section]

\numberwithin{equation}{section} 
\theoremstyle{plain}
\theoremstyle{plain}
\theoremstyle{definition}
\newtheorem{defn}[thm]{Definition}

\theoremstyle{plain}
\newtheorem{prop}[thm]{Proposition}

\newtheorem{lem}[thm]{Lemma}

\newtheorem*{cor*}{Corollary}
\newtheorem*{conj*}{Conjecture}
\newtheorem*{thm*}{Theorem}
\newtheorem*{lem*}{Lemma}

\newcommand{\bl}{\begin{lem}}
\newcommand{\el}{\end{lem}}
\newcommand{\bml}{\begin{multline}}
\newcommand{\eml}{\end{multline}}
\newcommand{\beq}{\begin{equation}}
\newcommand{\eeq}{\end{equation}}
\newcommand{\bp}{\begin{prop}}
\newcommand{\ep}{\end{prop}}
\newcommand{\bd}{\begin{defn}}
\newcommand{\ed}{\end{defn}}
\newcommand{\pf}{\begin{proof}}
\newcommand{\epf}{\end{proof}}

\newcommand{\F}{\mathcal{F}}

\begin{document}

\title{Effective sup-norm bounds on average \\ for cusp forms of even weight}
\author{J.S.~Friedman\footnote{The views expressed in this article are the author's own and not those 
of the U.S. Merchant Marine Academy, the Maritime Administration, the Department of Transportation, 
or the United States government.} \and J.~Jorgenson\footnote{The second named author acknowledges 
support from numerous PSC-CUNY grants.} \and J.~Kramer\footnote{The third named author 
acknowledges support from the DFG Graduate School \emph{Berlin Mathematical School}.}}

\maketitle

\begin{abstract}
\noindent
Let $\Gamma\subset\mathrm{PSL}_{2}(\mathbb{R})$ be a Fuchsian subgroup of the first kind acting 
on the upper half-plane $\mathbb{H}$. Consider the $d_{2k}$-dimensional space of cusp forms 
$\mathcal{S}_{2k}^{\Gamma}$ of weight $2k$ for $\Gamma$, and let $\{f_{1},\ldots,f_{d_{2k}}\}$ be 
an orthonormal basis of $\mathcal{S}_{2k}^{\Gamma}$ with respect to the Petersson inner product. 
In this paper we will give \emph{effective} upper and lower bounds for the supremum of the quantity
$S_{2k}^{\Gamma}(z):=\sum_{j=1}^{d_{2k}}\vert f_{j}(z)\vert^{2}\,\mathrm{Im}(z)^{2k}$ as $z$ ranges
through $\mathbb{H}$.
\end{abstract}

\section{Introduction}
\label{section-1}

\subsection{Statement of the main results}
\label{subsection-1.1}
Let $\Gamma\subset\mathrm{PSL}_{2}(\mathbb{R})$ be a Fuchsian subgroup of the first kind acting 
by fractional linear transformations on the upper half-plane $\mathbb{H}$, so that the quotient space 
$\Gamma\backslash\mathbb{H}$ has finite volume. For any integer $k\in\mathbb{N}_{\geq 1}$, we 
then consider the space $\mathcal{S}_{2k}^{\Gamma}$ of cusp forms of weight $2k$ for $\Gamma$, 
which is naturally equipped with the Petersson inner product. If $d_{2k}$ denotes the dimension of 
the $\mathbb{C}$-vector space $\mathcal{S}_{2k}^{\Gamma}$, we let $\{f_{1},\ldots,f_{d_{2k}}\}$ be 
an orthonormal basis of $\mathcal{S}_{2k}^{\Gamma}$. The purpose of this article is to determine 
\emph{effective} upper and lower bounds for the supremum of the quantity
\begin{align}
\label{formula-1.1}
S_{2k}^{\Gamma}(z):=\sum\limits_{j=1}^{d_{2k}}\vert f_{j}(z)\vert^{2}\,\mathrm{Im}(z)^{2k}
\end{align}
as $z$ ranges through $\mathbb{H}$. Optimal sup-norm bounds for the quantity~\eqref{formula-1.1} 
have been given in the case $k=1$ in the articles~\cite{JK04} and~\cite{JK11}, and for $k\geq 1$ in 
the paper~\cite{FJK16}. However, the sup-norm bounds obtained in these papers are not effective. 
The present article completes our previous investigations by now providing \emph{effective} optimal 
sup-norm bounds for the quantity~\eqref{formula-1.1}.

The main results of the paper are summarized in the following three theorems. When $\Gamma$ 
is cocompact and torsionfree, we have the following result.  

\textbf{Theorem A.} \emph{Let $\Gamma$ be cocompact and torsionfree, and let $k\in\mathbb{N}_
{\geq 2}$. Then, the bounds
\begin{align*}
\frac{2k-1}{4\pi}\leq\sup_{z\in\mathbb{H}}S_{2k}^{\Gamma}(z)\leq\frac{2k-1}{4\pi}+C_{\Gamma}\,
e^{-\delta_{\Gamma}k}
\end{align*}
hold, where the constants $C_{\Gamma}$ and $\delta_{\Gamma}$ are effectively computable as
\begin{align*}
C_{\Gamma}=\frac{3\,e^{4\pi g_{\Gamma}/\ell_{\Gamma}}}{\pi(g_{\Gamma}-1)}\frac{(\cosh(\ell_
{\Gamma})+1)^{2}}{\log((\cosh(\ell_{\Gamma})+1)/2)}\quad\text{and}\quad\delta_{\Gamma}=\frac
{1}{2}\log\bigg(\frac{\cosh(\ell_{\Gamma})+1}{2}\bigg);
\end{align*}
here $g_{\Gamma}$ and $\ell_{\Gamma}$ denote the genus and the length of the shortest closed
geodesic on $\Gamma\backslash\mathbb{H}$, respectively.}

In the general case, when $\Gamma$ is cofinite, possibly with elliptic elements, we let $\mathcal
{F}$ be a closed and connected fundamental domain for $\Gamma$. For $Y>0$, we consider the 
neighborhoods $\F_{j}^{Y}$ of the $j$-th cusp of $\mathcal{F}$ ($j=1,\ldots,h$), and we let 
$\mathcal{F}_{Y}$ denote the closure of the complement of the union of the cuspidal neighborhoods 
in $\mathcal{F}$, i.e., we have the decomposition
\begin{align*}
\mathcal{F}=\mathcal{F}_{Y}\cup\mathcal{F}_{1}^{Y}\cup\ldots\cup\mathcal{F}_{h}^{Y}.
\end{align*}
We let $\mathcal{E}:=\{e_{1},\ldots,e_{n}\}\subset\mathcal{F}$ be the set of elliptic fixed points in 
$\mathcal{F}$ and denote the order of $e_{j}$ by $n_{j}$ ($j=1,\ldots,n$). Then, we have the 
following result.

\textbf{Theorem B.} \emph{Let $\Gamma$ be cofinite, $k\in\mathbb{N}_{\geq 2}$, $Y_{0}>0$, and 
$Y:=\max\{2Y_{0},16/\sqrt{15}\}$. Then, with the above notations, we have the following statements:}
\begin{enumerate}
\item[(1)]
\emph{There exist effectively computable constants $B_{Y}$ and $\sigma_{Y}$ (depending on $Y$) 
such that the upper bound
\begin{align*}
\sup_{z\in\mathcal{F}_{Y}}S_{2k}^{\Gamma}(z)\leq\frac{2k-1}{4\pi}\bigg(1+6\sum\limits_{e_{j}\in
\mathcal{E}}(n_{j}-1)\bigg)+12(2k-1)B_{Y}\,\sigma_{Y}^{-(k-2)}
\end{align*}
holds.}
\item[(2)]
\emph{If $Y\geq k/(2\pi)$, there exist effectively computable constants $B_{Y}$ and $\sigma_{Y}$ 
(depending on $Y$)  such that the upper bounds
\begin{align*}
\sup_{z\in\mathcal{F}^{Y}_{j}}S_{2k}^{\Gamma}(z)\leq\frac{2k-1}{4\pi}\bigg(1+6\sum\limits_{e_{j}\in
\mathcal{E}}(n_{j}-1)\bigg)+12(2k-1)B_{Y}\,\sigma_{Y}^{-(k-2)}
\end{align*}
hold for $j=1,\ldots,h$.}
\item[(3)]
\emph{If $Y<k/(2\pi)$, there exists an effectively computable constant $B_{k,Y_{0}}$ (depending on 
$k$ and $Y_{0}$) such that the upper bounds
\begin{align*}
\sup_{z\in\mathcal{F}^{Y}_{j}}S_{2k}^{\Gamma}(z)\leq\frac{2k-1}{4\pi}+\frac{3(2k-1)}{2\pi}\bigg(B_{k,
Y_{0}}+\frac{\sqrt{k}\,e^{5/4}}{\sqrt{\pi}}\bigg)
\end{align*}
hold for $j=1,\ldots,h$.}
\end{enumerate}
The constant $\sigma_{Y}$ is given in Definition~\ref{definition-3.3} and effectively bounded from 
below in Lemma~\ref{lemma-3.4}. The constants $B_{Y}$ and $B_{k,Y_{0}}$ are given in Definition
\ref{definition-3.5} and in Definition~\ref{definition-4.3}, respectively. 

As an example, we provide in Subsection~\ref{subsection-5.4} explicit upper bounds for the supremum 
of the quantity $S_{2k}^{\Gamma}(z)$ in the case when $\Gamma$ is the modular group. This example 
shows how the present investigations give rise to an algorithm to determine effective upper bounds
for the supremum of the quantity $S_{2k}^{\Gamma}(z)$ for more general Fuchsian subgroups 
$\Gamma$.

\textbf{Theorem C.} \emph{Let $\Gamma=\mathrm{PSL}_{2}(\mathbb{Z})$, $k\in\mathbb{N}$, and 
$Y=16/\sqrt{15}=4.132...$ Then, the upper bounds
\begin{align*}
S_{2k}^{\Gamma}(z)\leq
\begin{cases}
\displaystyle
\frac{31(2k-1)}{4\pi}+72(2k-1)1.014^{-(k-2)}\qquad\qquad\qquad\quad\,\,\,\text{if }k\geq 2,\,z\in
\mathcal{F}_{Y}, \\[2mm]
\displaystyle
\frac{31(2k-1)}{4\pi}+72(2k-1)1.014^{-(k-2)}\qquad\qquad\quad\,\text{if }2\leq k\leq 25,\,z\in\mathcal
{F}^{Y}_{1}, \\[2mm]
\displaystyle
\frac{2k-1}{4\pi}+\frac{3(2k-1)}{2\pi}\bigg(4^{-k+4}\bigg(\frac{k}{2\pi}\bigg)^{4}+\frac{\sqrt{k}\,e^{5/4}}
{\sqrt{\pi}}\bigg)\qquad\text{if }k\geq 26,\,z\in\mathcal{F}^{Y}_{1},
\end{cases}
\end{align*}
hold.}

In addition to the main results listed above, we also provide lower bounds for the supremum of the 
quantity $S_{2k}^{\Gamma}(z)$ when $k\in\mathbb{N}_{\geq 2}$. The corresponding bounds in the
case $k=1$ are discussed separately.

\subsection{Results related to this paper}
\label{subsection-1.3}
As mentioned above, the present article is the completion of our previous investigations~\cite{JK04}, 
\cite{JK11}, and~\cite{FJK16} to determine sup-norm bounds for cusp forms on average. Our primary 
motivation for these studies originated from the article~\cite{Sil86}, where the author determined the 
arithmetic degree of a modular parametrization of an elliptic curve defined over $\mathbb{Q}$ in 
terms of various quantities, including the Petersson norm of the cusp form of weight $2$ associated 
to this parametrization. Following Silverman's article, the authors of~\cite{AbbesUllmo} proved for
the congruence subgroups $\Gamma=\Gamma_{0}(N)$ ($N$ squarefree; $2,3\nmid N$) and $k=1$
that for any $\varepsilon>0$, one has the bound
\begin{align*}
\sup_{z\in\mathbb{H}}S_{2}^{\Gamma_{0}(N)}(z)=O(N^{2+\varepsilon}),
\end{align*}
which was improved in~\cite{MU98} to $O(N^{1+\varepsilon})$. In~\cite{JK04}, this bound was further
improved by establishing a bound of the form $O(1)$, which holds uniformly for all subgroups $\Gamma$ 
of finite index of a fixed Fuchsian subgroup $\Gamma_{0}$ of the first kind. The methodology of~\cite
{JK04} was to study and employ the long-time asymptotic behavior of the heat kernel associated to 
the hyperbolic Laplacian acting on smooth functions on $\Gamma\backslash\mathbb{H}$; in~\cite
{JK11} the main result of~\cite{JK04} was re-proved by relating it to special values of non-holomorphic 
elliptic, hyperbolic, as well as parabolic Eisenstein series. 

Again a heat kernel approach was developed in~\cite{FJK16} in order to obtain bounds for the supremum
of the quantity $S_{2k}^{\Gamma}(z)$ for Fuchsian subgroups $\Gamma$ of the first kind and for $k\in
\mathbb{N}_{\geq 1}$, ultimately leading to uniform sup-norm bounds with \emph{ineffective} constants. 
In the present paper, we exploit knowledge of the resolvent kernel in order to obtain uniform sup-norm 
bounds with \emph{effective} constants as stated in Theorem~A and Theorem~B. We mention here also
results related to this paper obtained in \cite{AMM16}.

In a different direction, numerous authors have studied sup-norm bounds for individual holomorphic
modular forms and non-holomorphic Maass forms. One of the main motivations for these investigations 
is the fact that a certain sup-norm bound for Maass forms implies the Lindel\"of hypothesis for certain 
$L$-functions (see~\cite[p.~178]{Iw}). We refer the reader to the articles~\cite{BH10},~\cite{HT13},~\cite
{Te15}, and the references therein for some of the most recent results. As discussed in~\cite{FJK16}, 
the results for sup-norm bounds on average and the results for bounds on individual sup-norms should 
be viewed as complementary since neither result implies the other.  

Finally, we mention that effective sup-norm bounds of the type considered in this paper continue to 
prove to be useful in arithmetic geometry as, for example, the articles~\cite{BF14},~\cite{Jav14},~\cite
{Jav16}, or~\cite{JK14} show. 

\subsection{Outline of the paper}
\label{subsection-1.3}
In the next section we setup the basic notation and recall the main results needed in the sequel of 
the paper. After providing a couple of technical lemmas, the main goal of the third section is to give 
upper bounds for certain Poincar\'e series, when $z$ is ranging through the compact domain~$\mathcal
{F}_{Y}$. In the fourth section we give upper bounds for the Poincar\' e series under consideration, 
when $z$ ranges through the cuspidal neighborhoods $\mathcal{F}^{Y}_{j}$. Based on the bounds
established in the third and fourth section, the main results of the paper, in particular Theorems A, 
B, and C, are proven in the fifth section. The last section, presented as an appendix, collects various
materials which support the understanding of the paper.

\section{Preliminaries}
\label{section-2}
In this section we setup the basic notation and recall the main results needed in the sequel of the 
paper.

\subsubsection*{Quotient spaces.}
\label{subsubsection-2.1.1}
Let $\Gamma\subset\mathrm{PSL}_{2}(\mathbb{R})$ be a Fuchsian subgroup of the first kind acting 
by fractional linear transformations on the upper half-plane $\mathbb{H}:=\{z\in\mathbb{C}\,\vert\,z=
x+iy\,,\,y>0\}$. Let $M$ be the quotient space $\Gamma\backslash\mathbb{H}$ and $g_{\Gamma}$ 
the genus of $M$. In the sequel, we identify $M$ with a fundamental domain $\mathcal{F}\subset
\mathbb{H}$ for the group $\Gamma$, which we assume to be closed and connected. We denote 
the set of  geodesic line segments which form the boundary $\partial\mathcal{F}$ of $\mathcal{F}$ 
by $\mathcal{S}$.

Denote by 
\begin{align*}
\mathcal{P}=\{p_{1},\ldots,p_{h}\}
\end{align*}
the set of cusps of $\mathcal{F}$. Let $\sigma_{j}\in\mathrm{PSL}_{2}(\mathbb{R})$ be a scaling 
matrix of the cusp $p_{j}$, that is, $p_{j}=\sigma_{j}i\infty$ with stabilizer subgroup $\Gamma_{p_
{j}}$ described as
\begin{align*}
\sigma_{j}^{-1}\Gamma_{p_{j}}\sigma_{j}=\bigg\langle\begin{pmatrix}1&1\\0&1\end{pmatrix}\bigg
\rangle\qquad(j=1,\ldots,h).
\end{align*}
For $Y>0$, we let $\mathcal{F}^{Y}_{j}\subset\mathcal{F}$ denote the neighborhood of the cusp 
$p_{j}$ characterized by
\begin{align*}
\sigma_{j}^{-1}\mathcal{F}^{Y}_{j}=\{z=x+iy\in\mathbb{H}\,\vert\,-1/2\leq x\leq 1/2,\,y\geq Y\}\qquad
(j=1,\ldots,h).
\end{align*}
With these notations, we define $\mathcal{F}_{Y}$ to be the closure of the complement  of the 
union $\mathcal{F}^{Y}_{1}\cup\ldots\cup\mathcal{F}^{Y}_{h}$ in $\mathcal{F}$, i.e., 
\begin{align*}
\mathcal{F}_{Y}:=\mathrm{cl}\big(\mathcal{F}\setminus\big(\mathcal{F}^{Y}_{1}\cup\ldots\cup
\mathcal{F}^{Y}_{h}\big)\big),
\end{align*}
which is compact; we note that $\mathcal{F}_{Y}=\mathcal{F}$, if $\Gamma$ is cocompact. We 
choose $0<m_{Y}<M_{Y}$ such that for all $z\in\mathcal{F}_{Y}$ the inequalities
\begin{align*}
m_{Y}\leq\mathrm{Im}(\sigma_{j}^{-1}z)\leq M_{Y}
\end{align*}
hold for all $j=1,\ldots,h$; we note that $m_{Y}$ and $M_{Y}$ depend on the choice of $Y$.

Denote by 
\begin{align*}
\mathcal{E}=\{e_{1},\ldots,e_{n}\}
\end{align*}
the set of of elliptic fixed points of $\mathcal{F}$, let $n_{j}$ denote the order of $e_{j}$, and let 
$\theta_{j}:=2\pi/n_{j}$ be the rotation angle of the corresponding primitive elliptic element ($j=1,
\ldots,n$). We put
\begin{align*}
\theta_{\Gamma}:=\min_{j=1,\ldots,n}\theta_{j};
\end{align*}
note that $\theta_{\Gamma}>0$.

\subsubsection*{Hyperbolic metric.}
\label{subsubsection-2.1.2}
We denote by $\mathrm{d}s^{2}_{\mathrm{hyp}}(z)$ the line element and by $\mu_{\mathrm{hyp}}
(z)$ the volume form corresponding to the hyperbolic metric on $\mathbb{H}$, which is compatible 
with the complex structure of $\mathbb{H}$ and has constant curvature equal to $-1$. Locally on 
$\mathbb{H}\setminus\Gamma\mathcal{E}$, we have
\begin{align*}
\mathrm{d}s^{2}_{\mathrm{hyp}}(z)=\frac{\mathrm{d}x^{2}+\mathrm{d}y^{2}}{y^{2}}\quad\textrm
{and}\quad\mu_{\mathrm{hyp}}(z)=\frac{\mathrm{d}x\wedge\mathrm{d}y}{y^{2}}\,.
\end{align*}
For $z,w\in\mathbb{H}$, we let $\mathrm{dist}_{\mathrm{hyp}}(z,w)$ denote the hyperbolic 
distance between these two points. For later purposes, it is useful to introduce the displacement
function
\begin{align}
\label{formula-2.1}
\sigma(z,w):=\cosh^{2}\bigg(\frac{\mathrm{dist}_{\mathrm{hyp}}(z,w)}{2}\bigg)=\frac{\vert z-\bar
{w}\vert^{2}}{4\,\mathrm{Im}(z)\mathrm{Im}(w)}\,.
\end{align}
We denote the hyperbolic length of the shortest closed geodesic on $M$ by $\ell_{\Gamma}$. 
Finally, for a domain $D\subset\mathbb{H}$, we denote its hyperbolic diameter by $\mathrm{diam}_
{\mathrm{hyp}}(D)$ and its hyperbolic volume by $\mathrm{vol}_{\mathrm{hyp}}(D)$.

\subsubsection*{Cusp forms of higher weights.} 
\label{subsubsection-2.1.3}
For $k\in\mathbb{N}_{\geq 1}$, we let $\mathcal{S}_{2k}^{\Gamma}$ denote the space of cusp 
forms of weight $2k$ for $\Gamma$, i.e., the space of holomorphic functions $f\colon\mathbb{H}
\longrightarrow\mathbb{C}$, which have the transformation behavior
\begin{align*}
f(\gamma z)=(cz+d)^{2k}f(z)
\end{align*}
for all $\gamma=\big(\begin{smallmatrix}a&b\\c&d\end{smallmatrix}\big)\in\Gamma$, and which 
vanish at all the cusps of $M$. The space $\mathcal{S}_{2k}^{\Gamma}$ is equipped with the
Petersson inner product
\begin{align*}
\langle f_{1},f_{2}\rangle:=\int\limits_{M}f_{1}(z)\overline{f_{2}(z)}\,y^{2k}\mu_{\mathrm{hyp}}(z)
\qquad(f_{1},f_{2}\in\mathcal{S}_{2k}^{\Gamma}).
\end{align*}
By letting $d_{2k}:=\dim_{\mathbb{C}}(\mathcal{S}_{2k}^{\Gamma})$ and choosing an orthonormal 
basis $\{f_{1},\ldots,f_{d_{2k}}\}$ of $\mathcal{S}_{2k}^{\Gamma}$, we define the quantity
\begin{align*}
S_{2k}^{\Gamma}(z):=\sum_{j=1}^{d_{2k}}\vert f_{j}(z)\vert^{2}\,y^{2k}.
\end{align*}
We note that the quantity $S_{2k}^{\Gamma}(z)$ is invariant under the action of the Fuchsian 
subgroup~$\Gamma$.

\subsubsection*{Maass forms of higher weights.} 
\label{subsubsection-2.1.4}
Following~\cite{Roelcke},~\cite{Fay}, or~\cite{Fischer}, we introduce for any $k\in\mathbb{N}_{\geq 
1}$, the space $\mathcal{V}_{k}^{\Gamma}$ of functions $\varphi\colon\mathbb{H}\longrightarrow
\mathbb{C}$, which have the transformation behavior
\begin{align*}
\varphi(\gamma z)=\bigg(\frac{cz+d}{c\bar{z}+d}\bigg)^{k}\varphi(z)=e^{2ik\,\textrm{arg}(cz+d)}
\varphi(z)
\end{align*}
for all $\gamma=\big(\begin{smallmatrix}a&b\\c&d\end{smallmatrix}\big)\in\Gamma$. For $\varphi
\in\mathcal{V}_{k}^{\Gamma}$, we set
\begin{align*}
\Vert\varphi\Vert^{2}:=\int\limits_{M}\vert\varphi(z)\vert^{2}\mu_{\mathrm{hyp}}(z),
\end{align*}
whenever it is defined. We then introduce the Hilbert space
\begin{align*}
\mathcal{H}_{k}^{\Gamma}:=\big\{\varphi\in\mathcal{V}_{k}^{\Gamma}\,\big\vert\,\Vert\varphi\Vert<
\infty\big\}
\end{align*}
equipped with the inner product
\begin{align*}
\langle\varphi_{1},\varphi_{2}\rangle:=\int\limits_{M}\varphi_{1}(z)\overline{\varphi_{2}(z)}\mu_
{\mathrm{hyp}}(z)\qquad(\varphi_{1},\varphi_{2}\in\mathcal{H}_{k}^{\Gamma}).
\end{align*}
The generalized Laplacian
\begin{align*}
\Delta_{k}:=-y^{2}\bigg(\frac{\partial^{2}}{\partial x^{2}}+\frac{\partial^{2}}{\partial y^{2}}\bigg)+2iky
\frac{\partial}{\partial x}
\end{align*}
acts on the smooth functions of $\mathcal{H}_{k}^{\Gamma}$ and extends to an essentially 
self-adjoint linear operator acting on a dense subspace of $\mathcal{H}_{k}^{\Gamma}$.

From~\cite{Fay} or~\cite{Fischer}, we quote that the eigenvalues for the equation
\begin{align*}
\Delta_{k}\varphi(z)=\lambda\varphi(z)\qquad(\varphi\in\mathcal{H}_{k}^{\Gamma})
\end{align*}
satisfy the inequality $\lambda\geq k(1-k)$.

Furthermore, if $\lambda=k(1-k)$, then the corresponding eigenfunction $\varphi$ is of the form 
$\varphi(z)=f(z)y^{k}$, where $f$ is a cusp form of weight $2k$ for $\Gamma$, i.e., we have an 
isomorphism of $\mathbb{C}$-vector spaces
\begin{align*} 
\ker\big(\Delta_{k}-k(1-k)\big)\cong\mathcal{S}_{2k}^{\Gamma}.
\end{align*}

\subsubsection*{Resolvent kernel.} 
\label{subsubsection-2.1.5}
From~\cite{Fischer}, we recall that for $k\in\mathbb{N}_{\geq 1}$, the resolvent kernel on $\mathbb
{H}$ associated to $\Delta_{k}$ is the integral kernel $G_{k}(s;z,w)$, which inverts the operator 
$\Delta_{k}-s(1-s)\mathrm{id}$, where $s\in W_{k}:=\mathbb{C}\setminus\{k-n,\,-k-n\,\vert\,n\in
\mathbb{N}\}$ and $z,w\in\mathbb{H}$.

When $z=w$, the resolvent kernel has a singularity, which we cancel out by considering the 
difference
\begin{align*}
G_{k}(s;z,w)-G_{k}(t;z,w)
\end{align*}
for $s,t\in W_{k}$. In particular, by taking $t=s+1$, we define for $s\in W_{k}$ and $z,w\in\mathbb{H}$
the function
\begin{align}
\label{formula-2.2}
g_{k}(s;z,w):=G_{k}(s;z,w)-G_{k}(s+1;z,w).
\end{align}
For an explicit formula for the resolvent kernel and further properties of the functions $G_{k}(s;z,w)$ 
and $g_{k}(s;z,w)$, we refer to Subsection~\ref{subsection-6.1} of the Appendix.

\subsubsection*{Spectral expansion.}
\label{subsubsection-2.1.6}
Let $\{\lambda_{j}\}_{j=0}^{\infty}$ be the set of eigenvalues of $\Delta_{k}$ acting on the Hilbert 
space $\mathcal{H}_{k}^{\Gamma}$, let $\{\varphi_{j}\}_{j\geq 0}$ denote the corresponding 
orthonormal basis of eigenfunctions, and let $E_{j}(\cdot,s')$ be the Eisenstein series associated 
to the cusp $p_{j}$ ($j=1,\ldots,h$); for the precise definition, see~\cite[\S~1.5]{Fischer}.

\begin{lem} 
\label{lemma-2.1}
Let $s,t\in W_{k}\cap\mathbb{R}$ such that $t>s>1$. Then, letting $\lambda:=s(1-s)$ and $\mu:=
t(1-t)$, we have
\begin{align} 
\notag
&\sum\limits_{j=0}^{\infty}\bigg(\frac{1}{\lambda_{j}-\lambda}-\frac{1}{\lambda_{j}-\mu}\bigg)\vert
\varphi_{j}(z)\vert^{2}+\frac{1}{4\pi}\sum\limits_{j=1}^{h}\int\limits_{-\infty}^{\infty}\bigg(\frac{1}{\frac
{1}{4}+r^{2}-\lambda}-\frac{1}{\frac{1}{4}+r^{2}-\mu}\bigg)\bigg\vert E_{j}\bigg(z,\frac{1}{2}+ir\bigg)
\bigg\vert^{2}\,\mathrm{d}r \\
\label{formula-2.3}
&=-\frac{1}{4\pi}\big(\psi(s+k)+\psi(s-k)-\psi(t+k)-\psi(t-k)\big)+\sum\limits_{\substack{\gamma\in
\Gamma\setminus\{\mathrm{id}\}\\\gamma=\big(\begin{smallmatrix}a&b\\c&d\end{smallmatrix}\big)}
}\bigg(\frac{c\bar{z}+d}{cz+d}\bigg)^{k}\bigg(\frac{z-\gamma\bar{z}}{\gamma z-\bar{z}}\bigg)^
{k}g_{k}(s;z,\gamma z);
\end{align}
here $\psi(\cdot)$ is the digamma function. Furthermore, all sums and integrals in the above formula 
converge uniformly for $s,t\in W_{k}$ as chosen above and $z\in\mathbb{H}$.
\end{lem}
\begin{proof}
For the proof, we refer to~\cite[p.~46, eq.~(2.1.4)]{Fischer}.
\end{proof}
Note that in~\cite{Fischer} subgroups of $\mathrm{SL}_{2}(\mathbb{R})$ are used instead of $\mathrm
{PSL}_{2}(\mathbb{R})$. Hence, the difference by a factor of $1/2$. Also, one needs to apply Dini's 
theorem to~\cite[p.~46, eq.~(2.1.4)]{Fischer} to obtain the uniform convergence.

\section{Effective estimates in the compact domain $\mathcal{F}_{Y}$} 
\label{section-3}
The main goal of this section is to give an upper bound for Poincar\'e series of the type
\begin{align*}
P_{k,\varepsilon}^{\Gamma}(z):=\sum\limits_{\gamma\in\Gamma\setminus\{\mathrm{id}\}}\sigma(z,
\gamma z)^{-(k+\varepsilon)}
\end{align*}
for $k\in\mathbb{N}_{\geq 2}$ and $\varepsilon>0$, when $z$ is ranging through the compact domain 
$\mathcal{F}_{Y}$. To obtain this bound, we first need a couple of technical lemmas.

\subsection{The displacement lemma}
\label{subsection-3.1}
In this subsection, we will give a lower bound for the displacement function $\sigma(z,\gamma z)$, 
when $\gamma\in\Gamma$ has no elliptic fixed points in the fundamental domain $\mathcal{F}$
and $z$ is ranging through the compact domain $\mathcal{F}_{Y}$. We start with the following 
definition.

\begin{defn}
\label{definition-3.1}
Recalling that $\mathcal{E}$ is the set of elliptic fixed points in the fundamental domain $\mathcal
{F}$ and that the boundary $\partial\mathcal{F}$ consists of the geodesic line segments in the set 
$\mathcal{S}$, we define the quantity
\begin{align}
\label{formula-3.1}
\mu_{\Gamma}:=\inf_{\substack{S\in\mathcal{S}\\e\in\mathcal{E}\setminus S}}\mathrm{dist}_{\mathrm
{hyp}}(S,e),
\end{align}
which will be bounded in the next lemma.
\end{defn}

\begin{lem} 
\label{lemma-3.2} 
With the notations of Definition~\ref{definition-3.1}, the inequality
\begin{align*}
\mu_{\Gamma}\leq\mathrm{dist}_{\mathrm{hyp}}(\mathcal{F},\Gamma\mathcal{E}\setminus\mathcal
{F})
\end{align*}
holds.
\end{lem}
\begin{proof}
We may assume that we have
\begin{align*}
\mathrm{dist}_{\mathrm{hyp}}(\mathcal{F},\Gamma\mathcal{E}\setminus\mathcal{F})=\mathrm
{dist}_{\mathrm{hyp}}(z,e),
\end{align*}
where $z\in\partial\mathcal{F}$ and $e\in\Gamma\mathcal{E}\setminus\mathcal{F}$. We show that 
the elliptic fixed point $e$ lies in a translate of $\mathcal{F}$, which borders $\mathcal{F}$. To show 
this, we assume the contrary. So, let $\mathcal{F}_{1}=\gamma_{1}\mathcal{F}$ be a translate of 
$\mathcal{F}$, which borders $\mathcal{F}$, and let $\mathcal{F}_{2}=\gamma_{2}\mathcal{F}$ be 
a translate of $\mathcal{F}$, which borders $\mathcal{F}_{1}$, but not $\mathcal{F}$, and containing 
$e$; here $\gamma_{1},\gamma_{2}\in\Gamma$. The geodesic line joining $z$ with $e$ of hyperbolic 
length $\mathrm{dist}_{\mathrm{hyp}}(z,e)$, then leaves $\mathcal{F}_{1}$ and enters $\mathcal{F}_
{2}$ in a point $z_{1}$. We thus obtain the bound
\begin{align*}
\mathrm{dist}_{\mathrm{hyp}}(\gamma_{1}^{-1}z_{1},\gamma_{1}^{-1}e)=\mathrm{dist}_{\mathrm
{hyp}}(z_{1},e)<\mathrm{dist}_{\mathrm{hyp}}(z,e).
\end{align*}
However, since $\gamma_{1}^{-1}z_{1}\in\mathcal{F}$ and $\gamma_{1}^{-1}e\in\Gamma\mathcal
{E}\setminus\mathcal{F}$, this leads to a contradiction, and hence we can assume that $e\in\mathcal
{F}_{1}$.

To complete the proof, we realize that $z\in S_{1}$ for some $S_{1}\in\mathcal{S}$, which necessarily 
has the property $S_{1}\subset\mathcal{F}\cap\mathcal{F}_{1}$. This shows that $\gamma_{1}^{-1}
z\in S$ for some suitable other $S\in\mathcal{S}$ (namely, $S=\gamma_{1}^{-1}S_{1}$). Furthermore, 
since $\gamma_{1}^{-1}e\in\mathcal{E}$, but $\gamma_{1}^{-1}e\notin S$ (otherwise, we would have 
$e\in S_{1}\subset\mathcal{F}$, which is not the case), we obtain 
\begin{align*}
\inf_{\substack{S\in\mathcal{S}\\e\in\mathcal{E}\setminus S}}\mathrm{dist}_{\mathrm{hyp}}(S,e)\leq
\mathrm{dist}_{\mathrm{hyp}}(\gamma_{1}^{-1}z,\gamma_{1}^{-1}e)=\mathrm{dist}_{\mathrm{hyp}}
(z,e)=\mathrm{dist}_{\mathrm{hyp}}(\mathcal{F},\Gamma\mathcal{E}\setminus\mathcal{F}),
\end{align*}
which proves the claimed inequality.
\end{proof}

\begin{defn}
\label{definition-3.3}
Let $\Gamma_{\mathcal{E}}:=\Gamma_{e_{1}} \cup\ldots\cup\Gamma_{e_{n}}$ and $Y>0$. Then,
we define the quantity
\begin{align}
\label{formula-3.2}
\sigma_{Y}:=\inf_{\substack{z\in\mathcal{F}_{Y}\\\gamma\in\Gamma\setminus\Gamma_{\mathcal
{E}}}}\sigma(z,\gamma z),
\end{align}
which will be bounded in the next lemma.
\end{defn}

\begin{lem} 
\label{lemma-3.4} 
With the notations of Section~\ref{section-2}, $\mu_{\Gamma}$ given in Definition~\ref{definition-3.1}, 
and $\sigma_{Y}$ given in Definition~\ref{definition-3.3}, the inequalities
\begin{align*}
\sigma_{Y}\geq\min\bigg\{\frac{\cosh(\ell_{\Gamma})+1}{2},\,\sinh^{2}(\mu_{\Gamma})\,\sin^{2}\bigg
(\frac{\theta_{\Gamma}}{2}\bigg)+1,\,\frac{m_{Y}^{2}}{4}+1,\,\frac{1}{4M_{Y}^{2}}+1\bigg\}\geq 1
\end{align*}
hold.
\end{lem}
\begin{proof}
Letting $z\in\mathcal{F}_{Y}$ and $\gamma\in\Gamma\setminus\Gamma_{\mathcal{E}}$, we need 
to distinguish and investigate the following four cases:

\emph{Case 1.} Let $\gamma$ be a hyperbolic element. Then we obviously have that $\mathrm
{dist}_{\mathrm{hyp}}(z,\gamma z)\geq\ell_{\Gamma}$, from which we conclude
\begin{align*}
\sigma(z,\gamma z)\geq\cosh^{2}\bigg(\frac{\ell_{\Gamma}}{2}\bigg)=\frac{\cosh(\ell_{\Gamma})+
1}{2}\,.
\end{align*}

\emph{Case 2.} Let $\gamma$ be an elliptic element associated to an elliptic fixed point $e\notin
\mathcal{F}$. Denoting by $\theta$ the rotation angle of the corresponding primitive elliptic element, 
we obtain from~\cite[p.~174, Theorem~7.35.1]{Beardon}
\begin{align*}
&\sinh\bigg(\frac{\mathrm{dist}_{\mathrm{hyp}}(z,\gamma z)}{2}\bigg)=\sinh\big(\mathrm{dist}_
{\mathrm{hyp}}(z,e)\big)\sin\bigg(\frac{\theta}{2}\bigg) \\
&\quad\geq\sinh\big(\mathrm{dist}_{\mathrm{hyp}}(\mathcal{F},\Gamma\mathcal{E}\setminus
\mathcal{F})\big)\sin\bigg(\frac{\theta_{\Gamma}}{2}\bigg)\geq\sinh(\mu_{\Gamma})\sin\bigg(\frac
{\theta_{\Gamma}}{2}\bigg),
\end{align*}
where the last inequality is justified by Lemma~\ref{lemma-3.2}. From this we immediately get
\begin{align*}
\sigma(z,\gamma z)=\sinh^{2}\bigg(\frac{\mathrm{dist}_{\mathrm{hyp}}(z,\gamma z)}{2}\bigg)+1
\geq\sinh^{2}(\mu_{\Gamma})\sin^{2}\bigg(\frac{\theta_{\Gamma}}{2}\bigg)+1.
\end{align*}

\emph{Case 3.} Let $\gamma$ be a parabolic element associated to a cusp $p\notin\mathcal{P}$. 
Then, we have $\gamma\in\Gamma_{p}$ and there exists a $\gamma'\in\Gamma$ such that $p=
\gamma'p_{j}$ for some $j\in\{1,\ldots,h\}$. For the stabilizer subgroup $\Gamma_{p}$, we then 
find
\begin{align*}
\gamma'^{-1}\Gamma_{p}\gamma'=\Gamma_{p_{j}},\quad\text{hence}\quad\sigma_{j}^{-1}
\gamma'^{-1}\Gamma_{p}\gamma'\sigma_{j}=\sigma_{j}^{-1}\Gamma_{p_{j}}\sigma_{j}=\bigg
\langle\begin{pmatrix}1&1\\0&1\end{pmatrix}\bigg\rangle\,.
\end{align*}
Therefore, by setting 
\begin{align*}
\delta:=\sigma_{j}^{-1}\gamma'^{-1}\sigma_{j}=\begin{pmatrix}a&b\\c&d\end{pmatrix}\in\sigma_
{j}^{-1}\Gamma\sigma_{j},
\end{align*}
we find
\begin{align*}
\delta\sigma_{j}^{-1}\gamma\sigma_{j}\delta^{-1}=\begin{pmatrix}1&n\\0&1\end{pmatrix} 
\end{align*}
with some $n\in\mathbb{Z}$. Letting $z':=\sigma^{-1}_{j}z$, we now compute
\begin{align*}
\sigma(z,\gamma z)&=\cosh^{2}\bigg(\frac{\mathrm{dist}_{\mathrm{hyp}}(z,\gamma z)}{2}\bigg)=
\cosh^{2}\bigg(\frac{\mathrm{dist}_{\mathrm{hyp}}(\delta\sigma_{j}^{-1}z,\delta\sigma_{j}^{-1}
\gamma z)}{2}\bigg) \\[2mm]
&=\cosh^{2}\bigg(\frac{\mathrm{dist}_{\mathrm{hyp}}(\delta z',\delta\sigma_{j}^{-1}\gamma\sigma_
{j}\delta^{-1}\delta z')}{2}\bigg)=\cosh^{2}\bigg(\frac{\mathrm{dist}_{\mathrm{hyp}}(\delta z',\delta 
z'+n)}{2}\bigg) \\[2mm]
&=\frac{\vert\delta z'-\delta\bar{z}'-n\vert^{2}}{4\,\mathrm{Im}(\delta z')^{2}}=\frac{\frac{4\,\mathrm
{Im}(z')^{2}}{\vert cz'+d\vert^{4}}+n^{2}}{4\,\mathrm{Im}(\delta z')^{2}}=\frac{n^{2}\,\vert cz'+d\vert^
{4}}{4\,\mathrm{Im}(z')^{2}}+1.
\end{align*}
Taking into account that $c\neq 0$ (since otherwise we would have $p\in\mathcal{P}$), Shimizu's
lemma gives the bound $\vert c\vert\geq 1$. Thus, the latter quantity can be bounded as
\begin{align*}
\sigma(z,\gamma z)&\geq\frac{\big((c\,\mathrm{Re}(z')+d)^{2}+c^{2}\,\mathrm{Im}(z')^{2}\big)^{2}}
{4\,\mathrm{Im}(z')^{2}}+1\geq\frac{c^{4}\,\mathrm{Im}(z')^{2}}{4}+1\geq\frac{m_{Y}^{2}}{4}+1.
\end{align*}

\emph{Case 4.} Let $\gamma$ be a parabolic element associated to a cusp $p_{j}\in\mathcal{P}$. 
By proceding as in the previous case with $\gamma'=\mathrm{id}$ and hence $\delta=\mathrm{id}$, 
we have
\begin{align*}
\sigma_{j}^{-1}\gamma\sigma_{j}=\begin{pmatrix}1&n\\0&1\end{pmatrix} 
\end{align*}
with some $n\in\mathbb{Z}$. Letting $z':=\sigma_{j}^{-1}z$, we compute as in the previous case
\begin{align*}
\sigma(z,\gamma z)=\cosh^{2}\bigg(\frac{\mathrm{dist}_{\mathrm{hyp}}(z,\gamma z)}{2}\bigg)= 
\frac{\vert z'-\bar{z}'-n\vert^{2}}{4\,\mathrm{Im}(z')^{2}}=\frac{n^{2}}{4\,\mathrm{Im}(z')^{2}}+1\geq
\frac{1}{4M_{Y}^{2}}+1.
\end{align*}
This completes the proof of the lemma, observing that the second claimed inequality is clear.
\end{proof}

\subsection{Upper bounds for Poincar\'e series in the compact domain $\mathcal{F}_{Y}$}
\label{subsection-3.2}
In this subsection, we will give an upper bound for the Poincar\'e series $P_{k,\varepsilon}^{\Gamma}
(z)$ for $k\in\mathbb{N}_{\geq 2}$ and $\varepsilon>0$, when $z$ is ranging through the compact 
domain $\mathcal{F}_{Y}$. We start with the following definition.

\begin{defn}
\label{definition-3.5}
Recalling that $\mathrm{diam}_{\mathrm{hyp}}(\mathcal{F}_{Y})$ denotes the hyperbolic diameter 
of $\mathcal{F}_{Y}$, we define the quantity
\begin{align}
\label{formula-3.3}
B_{Y}:=e^{\mathrm{diam}_{\mathrm{hyp}}(\mathcal{F}_{Y})/2}/\mathrm{vol}_{\mathrm{hyp}}(\mathcal
{F}_{Y}),
\end{align}
which will be useful in the next lemma.
\end{defn}

\begin{lem}
\label{lemma-3.6}
For $z\in\mathcal{F}_{Y}$ and $r\geq 1$, let $\pi_{\mathcal{F}_{Y}}(z,r)$ denote the counting 
function
\begin{align*}
\pi_{\mathcal{F}_{Y}}(z,r):=\#\big\{\gamma\in\Gamma\,\vert\,\sigma(z,\gamma z)\leq r\big\}.
\end{align*}
Then, with $B_{Y}$ given in Definition~\ref{definition-3.5}, the upper bound
\begin{align*}
\pi_{\mathcal{F}_{Y}}(z,r)\leq 4\pi\,B_{Y}r
\end{align*}
holds.
\end{lem}
\begin{proof}
By choosing $\rho\geq 0$ such that $r=\cosh^{2}(\rho/2)$, we have
\begin{align*}
\pi_{\mathcal{F}_{Y}}(z,r)=\#\big\{\gamma\in\Gamma\,\vert\,\mathrm{dist}_{\mathrm{hyp}}(z,\gamma 
z)\leq\rho\big\}.
\end{align*}
Fix now $z_{0}\in\mathcal{F}_{Y}$ such that the disk $B_{z_{0}}(r_{0})$ of hyperbolic radius $r_{0}:=
\mathrm{diam}_{\mathrm{hyp}}(\mathcal{F}_{Y})/2$ centered at $z_{0}$ covers $\mathcal{F}_{Y}$. 
Then, as $\gamma$ runs through the set $\{\gamma\in\Gamma\,\vert\,\mathrm{dist}_{\mathrm{hyp}}
(z,\gamma z)\leq\rho\}$, the translates $\gamma\mathcal{F}_{Y}$ disjointly cover parts of the disk 
$B_{z_{0}}(r_{0}+\rho)$ of hyperbolic radius $r_{0}+\rho$ centered at $z_{0}$. This leads to the 
upper bound
\begin{align*}
\pi_{\mathcal{F}_{Y}}(z,r)\cdot\mathrm{vol}_{\mathrm{hyp}}(\mathcal{F}_{Y})&\leq\mathrm{vol}_
{\mathrm{hyp}}\big(B_{z_{0}}(r_{0}+\rho)\big)=4\pi\sinh^{2}\bigg(\frac{r_{0}+\rho}{2}\bigg) \\
&\leq 4\pi\cosh^{2}\bigg(\frac{r_{0}+\rho}{2}\bigg)\leq 4\pi\,e^{r_{0}}\cosh^{2}\bigg(\frac{\rho}{2}
\bigg).
\end{align*}
This immediately implies the claimed upper bound recalling that $r_{0}=\mathrm{diam}_{\mathrm
{hyp}}(\mathcal{F}_{Y})/2$ and $r=\cosh^{2}(\rho/2)$.
\end{proof}

\begin{lem}
\label{lemma-3.7}
Let $Y>0$ and $\delta>1$. Then, with $B_{Y}$ given in Definition~\ref{definition-3.5}, the upper 
bound
\begin{align*}
\sum\limits_{\gamma\in\Gamma}\sigma(z,\gamma z)^{-\delta}\leq 4\pi\,B_{Y}\frac{\delta}{\delta-1}
\end{align*}
holds for $z\in\mathcal{F}_{Y}$.
\end{lem}
\begin{proof}
Letting $R>1$ and rewriting the Poincar\'e series under consideration as a Stieltjes integral using 
the counting function $\pi_{\mathcal{F}_{Y}}(z,r)$ from Lemma~\ref{lemma-3.6}, we get after
integrating by parts
\begin{align*}
\sum\limits_{\substack{\gamma\in\Gamma\\\sigma(z,\gamma z)\leq R}}\sigma(z,\gamma z)^{-\delta}
=\int\limits_{1}^{R}r^{-\delta}\,\mathrm{d}\pi_{\mathcal{F}_{Y}}(z,r)=r^{-\delta}\,\pi_{\mathcal{F}_{Y}}
(z,r)\bigg\vert_{1}^{R}+\delta\int\limits_{1}^{R}r^{-\delta-1}\,\pi_{\mathcal{F}_{Y}}(z,r)\,\mathrm{d}r.
\end{align*}
Using Lemma~\ref{lemma-3.6}, we find upon setting $\widetilde{B}_{Y}:=4\pi B_{Y}$ the bound
\begin{align*}
\sum\limits_{\substack{\gamma\in\Gamma\\\sigma(z,\gamma z)\leq R}}\sigma(z,\gamma z)^{-\delta}
\leq R^{-\delta}\widetilde{B}_{Y}\,R+\delta\int\limits_{1}^{R}r^{-\delta-1}\widetilde{B}_{Y}\,r\,\mathrm
{d}r=\widetilde{B}_{Y}R^{-\delta+1}+\widetilde{B}_{Y}\,\delta\bigg(\frac{R^{-\delta+1}}{-\delta+1}-
\frac{1}{-\delta+1}\bigg).
\end{align*}
Letting $R\rightarrow\infty$, we thus obtain the upper bound
\begin{align*}
\sum\limits_{\gamma\in\Gamma}\sigma(z,\gamma z)^{-\delta}\leq 4\pi\,B_{Y}\frac{\delta}{\delta-1}\,,
\end{align*}
as claimed.
\end{proof}

\begin{prop}
\label{proposition-3.8}
Let $k\in\mathbb{N}_{\geq 2}$, $\varepsilon>0$, and $Y>0$. Then, with $\sigma_{Y}$ given in
Definition~\ref{definition-3.3} and $B_{Y}$ given in Definition~\ref{definition-3.5}, the upper bound
\begin{align*}
\sum\limits_{\gamma\in\Gamma\setminus\{\mathrm{id}\}}\sigma(z,\gamma z)^{-(k+\varepsilon)}\leq
4\pi\,\frac{2+\varepsilon}{1+\varepsilon}\,B_{Y}\,\sigma_{Y}^{-(k-2)}+\sum\limits_{e_{j}\in\mathcal{E}}
(n_{j}-1)
\end{align*}
holds for $z\in\mathcal{F}_{Y}$.
\end{prop}
\begin{proof}
Recalling that $\Gamma_{\mathcal{E}}=\Gamma_{e_{1}} \cup\ldots\cup\Gamma_{e_{n}}$, we have 
the decomposition
\begin{align*}
\sum\limits_{\gamma\in\Gamma\setminus\{\mathrm{id}\}}\sigma(z,\gamma z)^{-(k+\varepsilon)}=\sum
\limits_{\gamma\in\Gamma\setminus\Gamma_{\mathcal{E}}}\sigma(z,\gamma z)^{-(k+\varepsilon)}+
\sum\limits_{\gamma\in\Gamma_{\mathcal{E}}\setminus\{\mathrm{id}\}}\sigma(z,\gamma z)^{-(k+
\varepsilon)}.
\end{align*}
Since $\sigma(z,\gamma z)\geq\sigma_{Y}\geq 1$ for $z\in\mathcal{F}_{Y}$ and $\gamma\in\Gamma
\setminus\Gamma_{\mathcal{E}}$, and since $k\in\mathbb{N}_{\geq 2}$, Lemma~\ref{lemma-3.7} 
allows to bound the first summand as
\begin{align*}
&\sum\limits_{\gamma\in\Gamma\setminus\Gamma_{\mathcal{E}}}\sigma(z,\gamma z)^{-(k+\varepsilon)}
=\sum\limits_{\gamma\in\Gamma\setminus\Gamma_{\mathcal{E}}}\sigma(z,\gamma z)^{-(k-2)}\,\sigma
(z,\gamma z)^{-(2+\varepsilon)} \\[2mm]
&\qquad\leq\sigma_{Y}^{-(k-2)}\sum\limits_{\gamma\in\Gamma\setminus\Gamma_{\mathcal{E}}}\sigma
(z,\gamma z)^{-(2+\varepsilon)}\leq\sigma_{Y}^{-(k-2)}\,4\pi\,B_{Y}\,\frac{2+\varepsilon}{1+\varepsilon}\,.
\end{align*}
Since $\sigma(z,\gamma z)\geq 1$ for $z\in\mathcal{F}_{Y}$ and $\gamma\in\Gamma_{\mathcal{E}}$,
we easily estimate the second summand as
\begin{align*}
&\sum\limits_{\gamma\in\Gamma_{\mathcal{E}}\setminus\{\mathrm{id}\}}\sigma(z,\gamma z)^{-(k+
\varepsilon)}\leq\sum\limits_{e_{j}\in\mathcal{E}}(n_{j}-1).
\end{align*}
This completes the proof of the proposition.
\end{proof}

\section{Effective estimates in the cuspidal neighborhoods $\mathcal{F}^{Y}_{j}$}
\label{section-4}
The main goal of this section is to give an upper bound for the Poincar\' e series $P_{k,\varepsilon}^
{\Gamma}(z)$ for $k\in\mathbb{N}_{\geq 2}$ and $\varepsilon>0$, when $z$ ranges through the 
cuspidal neighborhoods $\mathcal{F}^{Y}_{j}$. It will turn out that we can restrict ourselves to the
case when $Y<k/(2\pi)$.

\subsection{A lemma of Faddeev}
\label{subsection-4.1}
In this subsection, we first show that bounding $S_{2k}^{\Gamma}(z)$ in the cuspidal neighborhoods 
$\mathcal{F}^{Y}_{j}$ can be reduced to estimating this quantity in suitable compact sets depending
on $Y\geq k/(2\pi)$ or $Y<k/(2\pi)$. Then, we will prove a lemma due to L.D.~Faddeev~\cite{Faddeev}, 
which will be crucial for the next subsection.

\begin{lem} 
\label{lemma-4.1}
Let $k\in\mathbb{N}_{\geq 1}$. Then, for $j=1,\ldots,h$, we have the following two statements:
\begin{itemize}
\item[{\rm(1)}]
For $Y\geq k/(2\pi)$, the inequality
\begin{align*}
\sup_{z\in\mathcal{F}^{Y}_{j}}S_{2k}^{\Gamma}(z)\leq\sup_{z\in\mathcal{F}_{Y}}S_{2k}^{\Gamma}(z)
\end{align*}
holds.
\item[{\rm(2)}]
For $Y<k/(2\pi)$, the equality
\begin{align*}
\sup_{z\in\mathcal{F}^{Y}_{j}}S_{2k}^{\Gamma}(z)=\sup_{z\in\mathrm{cl}(\mathcal{F}^{Y}_{j}\setminus
\mathcal{F}^{k/(2\pi)}_{j})}S_{2k}^{\Gamma}(z)
\end{align*}
holds; here $\mathrm{cl}(\,\cdot\,)$ refers to the topological closure.
\end{itemize}
\end{lem}
\begin{proof}
(1) Without loss of generality, we may assume that $p_{j}=i\infty$ with scaling matrix $\sigma_{j}=
\mathrm{id}$, so that we have
\begin{align*}
\mathcal{F}^{Y}_{j}=\{z=x+iy\in\mathbb{H}\,\vert\,-1/2\leq x\leq 1/2,\,y\geq Y\}.
\end{align*}
By then focussing on a single cusp form $f\in\mathcal{S}_{2k}^{\Gamma}$ with Fourier expansion
\begin{align*}
f(z)=\sum\limits_{n=1}^{\infty}a_{n}e^{2\pi inz},
\end{align*}
we have to estimate the expression
\begin{align*}
\vert f(z)\vert^{2}\,y^{2k}=\bigg\vert\frac{f(z)}{e^{2\pi i z}}\bigg\vert^{2}\,e^{-4\pi y}y^{2k}
\end{align*}
in the strip $\mathcal{F}^{Y}_{j}$. Since the function $\vert f(z)/e^{2\pi i z}\vert^{2}$ is bounded and 
subharmonic in $\mathcal{F}^{Y}_{j}$, the strong maximum principle for subharmonic functions implies 
that its maximum occurs when $y=Y$.

Next we consider the function $h_{k}(y):=e^{-4\pi y}y^{2k}$ for $y>0$. Elementary calculus shows that
\begin{align*}
h'_{k}(y)=2ke^{-4\pi y}y^{2k-1}-4\pi e^{-4\pi y}y^{2k},
\end{align*}
so then $h_{k}(y)$ achieves its maximum when $y=k/(2\pi)\leq Y$. Therefore, by the monotonicity of 
the function $h_{k}(y)$, we find that
\begin{align*}
\max_{z\in\mathcal{F}^{Y}_{j}}\vert f(z)\vert^{2}\,y^{2k}=\max_{\substack{-1/2\leq x\leq 1/2\\y=Y}}\vert 
f(z)\vert^{2}\,y^{2k}\leq\max_{z\in\mathcal{F}_{Y}}\vert f(z)\vert^{2}\,y^{2k},
\end{align*}
which proves the first part of the claim.

(2) Since $Y<k/(2\pi)$, we have the proper decomposition
\begin{align*}
\mathcal{F}^{Y}_{j}=\mathcal{F}^{k/(2\pi)}_{j}\cup\big(\mathcal{F}^{Y}_{j}\setminus\mathcal{F}^{k/(2
\pi)}_{j}\big).
\end{align*}
Proceeding as in (1), we are then led to the equality
\begin{align*}
\max_{z\in\mathcal{F}^{k/(2\pi)}_{j}}\vert f(z)\vert^{2}\,y^{2k}=\max_{\substack{-1/2\leq x\leq 1/2\\y=
k/(2\pi)}}\vert f(z)\vert^{2}\,y^{2k}.
\end{align*}
From this we immediately conclude that
\begin{align*}
\max_{z\in\mathcal{F}^{Y}_{j}}\vert f(z)\vert^{2}\,y^{2k}=\max_{z\in\mathrm{cl}(\mathcal{F}^{Y}_{j}
\setminus\mathcal{F}^{k/(2\pi)}_{j})}\vert f(z)\vert^{2}\,y^{2k},
\end{align*}
which proves the second part of the claim.
\end{proof}

The next lemma is due to L.D.~Faddeev~\cite{Faddeev}; for its proof, we follow~\cite[p.~307]{Lang2}. 

\begin{lem} 
\label{lemma-4.2}
Let $p_{j}$ be a cusp of $\mathcal{F}$ with scaling matrix $\sigma_{j}$, $z_{0}=x+iy_{0}\in\mathbb
{H}$, and $\delta_{1}>0$. Then, the inequality
\begin{align*}
\sum\limits_{\gamma\in\Gamma\setminus\Gamma_{p_{j}}}\sigma(\sigma_{j}z,\gamma\sigma_{j}z)^
{-\delta_{2}}\leq\bigg(\frac{64}{15}\bigg)^{\delta_{2}-\delta_{1}-1}y_{0}^{-2\delta_{1}-2}y^{-2\delta_{2}+
4\delta_{1}+4}\sum\limits_{\gamma\in\Gamma\setminus\Gamma_{p_{j}}}\sigma(\sigma_{j}z_{0},
\gamma\sigma_{j}z_{0})^{-\delta_{1}-1}
\end{align*}
holds for $z=x+iy\in\mathbb{H}$ with $y\geq 2y_{0}$ and $\delta_{2}\geq\delta_{1}+1$.
\end{lem}
\begin{proof}
Since we have $\sigma(\sigma_{j}z,\gamma\sigma_{j}z)=\sigma(z,\sigma_{j}^{-1}\gamma\sigma_{j}
z)$ and
\begin{align*}
\sigma_{j}^{-1}\Gamma_{p_{j}}\sigma_{j}=\bigg\langle\begin{pmatrix}1&1\\0&1\end{pmatrix}\bigg
\rangle,
\end{align*}
we may assume without loss of generality that $p_{j}=i\infty$ and $\sigma_{j}=\mathrm{id}$. For 
any
\begin{align*}
\gamma=\begin{pmatrix}a&b\\c&d\end{pmatrix}\in\Gamma\setminus\Gamma_{i\infty},
\end{align*}
we then have $\vert c\vert\geq 1$ by Shimizu's lemma. Using 
\begin{align*}
u(z,w):=\sigma(z,w)-1=\sinh^{2}\bigg(\frac{\mathrm{dist}_{\mathrm{hyp}}(z,w)}{2}\bigg)=\frac{\vert 
z-w\vert^{2}}{4\,\mathrm{Im}(z)\mathrm{Im}(w)}
\end{align*}
with $w=\gamma z$, a direct calculation shows that
\begin{align*} 
4y^{2}u(z,\gamma z)&=\vert cz^{2}+dz-az-b\vert^{2} \\
&=(cx^{2}+dx-ax-b)^{2}+(cx+d)^{2}y^{2}+(cx -a)^{2}y^{2}+c^{2}y^{4}-2y^{2};
\end{align*}
a similar equation holds for $z_{0}=x+iy_{0}.$ Hence, recalling that $y\geq 2y_{0}$, yields the
inequality
\begin{align*}
4y^{2}u(z,\gamma z)\geq 4y_{0}^{2}u(z_{0},\gamma z_{0})+c^{2}y^{4}-c^{2}y_{0}^{4}-2y^{2}+2
y_{0}^{2}.
\end{align*}
Next, adding $4y^{2}+4y_{0}^{2}$ to both sides, gives (again, using $y\geq 2y_{0}$)
\begin{align*}
4y^{2}\sigma(z,\gamma z)\geq 4y_{0}^{2}\sigma(z_{0},\gamma z_{0})+c^{2}(y^{4}-y_{0}^{4})+2
y^{2}-2y_{0}^{2}\geq 4y_{0}^{2}\sigma(z_{0},\gamma z_{0})+\frac{15}{16}c^{2}y^{4}.
\end{align*}
After dividing both sides by $y^{4}$, we obtain
\begin{align*}
\frac{4}{y^{2}}\sigma(z,\gamma z)\geq\frac{4y_{0}^{2}}{y^{4}}\sigma(z_{0},\gamma z_{0})+\frac
{15}{16}c^{2}.
\end{align*}
Next, multiply both sides by $16/15$ and use $\vert c\vert\geq 1$ to get
\begin{align*}
\frac{64}{15y^{2}}\sigma(z,\gamma z)\geq\frac{64y_{0}^{2}}{15y^{4}}\sigma(z_{0},\gamma z_{0})
+c^{2}\geq 1.
\end{align*}
Since both sides are at least one, we obtain after exponentiating with $\delta_{2}\geq\delta_{1}
+1>1$,
\begin{align*}
\bigg(\frac{64}{15y^{2}}\sigma(z,\gamma z)\bigg)^{\delta_{2}}\geq\bigg(\frac{64y_{0}^{2}}{15y^
{4}}\sigma(z_{0},\gamma z_{0})\bigg)^{\delta_{1}+1}.
\end{align*}
Rearranging terms leads to the inequality
\begin{align*}
\sigma(z,\gamma z)^{-\delta_{2}}\leq\bigg(\frac{64}{15}\bigg)^{\delta_{2}-\delta_{1}-1}y_{0}^{-2
\delta_{1}-2}y^{-2\delta_{2}+4\delta_{1}+4}\sigma(z_{0},\gamma z_{0})^{-\delta_{1}-1},
\end{align*}
which proves the claimed inequality after taking the sum over $\gamma\in\Gamma\setminus
\Gamma_{i\infty}$.
\end{proof}

\subsection{Upper bounds for Poincar\'e series in the cuspidal neighborhoods $\mathcal{F}^{Y}_
{j}$}
\label{subsection-4.2}
In this subsection, we will apply Faddeev's lemma to obtain an upper bound for the Poincar\' e 
series $P_{k,\varepsilon}^{\Gamma}(z)$ for $k\in\mathbb{N}_{\geq 2}$ and $\varepsilon>0$, 
when $z$ ranges through the cuspidal neighborhoods $\mathcal{F}^{Y}_{j}$ with $Y<k/(2\pi)$.
We start with the following definition.

\begin{defn} 
\label{definition-4.3}
Let $k\in\mathbb{N}_{\geq 1}$, $\varepsilon>0$, and $Y_{0}>0$. Then, we define the quantities
\begin{align}
\label{formula-4.1}
B_{k,Y_{0}}(\varepsilon):=\pi\,Y_{0}^{-4-2\varepsilon}\,B_{Y_{0}}\,4^{-k+3}\,\frac{2+\varepsilon}
{1+\varepsilon}\bigg(\frac{k}{2\pi}\bigg)^{4+2\varepsilon},
\end{align}
and
\begin{align}
\label{formula-4.2}
B_{k,Y_{0}}:=\lim_{\varepsilon\to 0}B_{k,Y_{0}}(\varepsilon)=2\pi\,Y_{0}^{-4}\,B_{Y_{0}}\,4^{-k+3}
\bigg(\frac{k}{2\pi}\bigg)^{4},
\end{align}
which will be useful for the next lemma.
\end{defn} 

\begin{lem}  
\label{lemma-4.4}
Let $k\in\mathbb{N}_{\geq 2}$, $\varepsilon>0$, $Y_{0}>0$, and $Y:=\max\{2Y_{0},16/\sqrt{{15}}
\}$; assume that $Y<k/(2\pi)$. Then, with $ B_{k,Y_{0}}(\varepsilon)$ given in Definition~\ref
{definition-4.3}, the upper bounds
\begin{align*}
\sum\limits_{\gamma\in\Gamma\setminus\Gamma_{p_{j}}}\sigma(z,\gamma z)^{-(k+\varepsilon)}
\leq B_{k,Y_{0}}(\varepsilon)\qquad(j=1,\ldots,h)
\end{align*}
hold for $z\in\mathrm{cl}(\mathcal{F}^{Y}_{j}\setminus\mathcal{F}^{k/(2\pi)}_{j})$.
\end{lem}
\begin{proof}
With the scaling matrix $\sigma_{j}$ of the cusp $p_{j}$, we define $z':=\sigma_{j}^{-1}z$. We then
employ Lemma~\ref{lemma-4.2} with $z'=x'+iy'$, $z_{0}:=x'+iY_{0}$ (note that $y'\geq 2Y_{0}$) and 
$\delta_{1}:=1+\varepsilon$, $\delta_{2}:=k+\varepsilon$ (note that $\delta_{2}\geq\delta_{1}+1$, 
since $k\in\mathbb{N}_{\geq 2}$), to get
\begin{align*} 
&\sum\limits_{\gamma\in\Gamma\setminus\Gamma_{p_{j}}}\sigma(z,\gamma z)^{-(k+\varepsilon)}
=\sum\limits_{\gamma\in\Gamma\setminus\Gamma_{p_{j}}}\sigma(\sigma_{j}z',\gamma\sigma_
{j}z')^{-(k+\varepsilon)} \\
&\qquad\leq\bigg(\frac{64}{15}\bigg)^{k-2}Y_{0}^{-4-2\varepsilon}y'^{-2k+8+2\varepsilon}\sum\limits_
{\gamma\in\Gamma\setminus\Gamma_{p_{j}}}\sigma(\sigma_{j}z_{0},\gamma\sigma_{j}z_{0})^{-(2+
\varepsilon)}.
\end{align*}
Since we have $Y\geq 2\cdot 8/\sqrt{15}$, we get $(64/15)^{k-2}\leq Y^{2k-4}/4^{k-2}$, which leads 
to the estimate
\begin{align*} 
\sum\limits_{\gamma\in\Gamma\setminus\Gamma_{p_{j}}}\sigma(z,\gamma z)^{-(k+\varepsilon)}
\leq\frac{Y^{2k-4}}{4^{k-2}}Y_{0}^{-4-2\varepsilon}Y^{-2k+4}\bigg(\frac{k}{2\pi}\bigg)^{4+2\varepsilon}
\sum\limits_{\gamma\in\Gamma\setminus\Gamma_{p_{j}}}\sigma(\sigma_{j}z_{0},\gamma\sigma_
{j}z_{0})^{-(2+\varepsilon)};
\end{align*}
here we used that $1<Y\leq y'\leq k/(2\pi)$. Observing now that we have by construction $\sigma_
{j}z_{0}\in\mathcal{F}_{Y_{0}}$, we can bound the latter sum from above by Lemma~\ref{lemma-3.7} 
as
\begin{align*}
\sum\limits_{\gamma\in\Gamma\setminus\Gamma_{p_{j}}}\sigma(\sigma_{j}z_{0},\gamma\sigma_
{j}z_{0})^{-(2+\varepsilon)}\leq 4\pi\,B_{Y_{0}}\,\frac{2+\varepsilon}{1+\varepsilon}.
\end{align*} 
All in all, this proves the claim.
\end{proof}

\begin{lem} 
\label{lemma-4.5}
Let $k\in\mathbb{N}_{\geq 1}$, $\varepsilon>0$, and $Y>0$; assume that $Y<k/(2\pi)$. Then, the
upper bounds
\begin{align*}
\sum\limits_{\gamma\in\Gamma_{p_{j}}\setminus\{\mathrm{id}\}}\sigma(z,\gamma z)^{-(k+\varepsilon)}
\leq\frac{k\,e^{5/4}}{\sqrt{\pi}\,\sqrt{k+\varepsilon}}\qquad(j=1,\ldots,h)
\end{align*}
hold for $z\in\mathrm{cl}(\mathcal{F}^{Y}_{j}\setminus\mathcal{F}^{k/(2\pi)}_{j})$.
\end{lem}
\begin{proof}
With the scaling matrix $\sigma_{j}$ of the cusp $p_{j}$, we define $z':=\sigma_{j}^{-1}z$. Recalling 
that
\begin{align*}
\sigma_{j}^{-1}\Gamma_{p_{j}}\sigma_{j}=\bigg\langle\begin{pmatrix}1&1\\0&1\end{pmatrix}\bigg
\rangle,
\end{align*}
we compute using $z'=x'+iy'$ that
\begin{align*}
&\sum\limits_{\gamma\in\Gamma_{p_{j}}\setminus\{\mathrm{id}\}}\sigma(z,\gamma z)^{-(k+\varepsilon)}
=\sum\limits_{\gamma\in\Gamma_{p_{j}}\setminus\{\mathrm{id}\}}\sigma(\sigma_{j}z',\gamma\sigma_
{j}z')^{-(k+\varepsilon)} \\
&\qquad=\sum\limits_{\substack{n\in\mathbb{Z}\\n\neq 0}}\sigma(z',z'+n)^{-(k+\varepsilon)}=2\sum
\limits_{n=1}^{\infty}\bigg(1+\bigg(\frac{n}{2y'}\bigg)^{2}\bigg)^{-(k+\varepsilon)}.
\end{align*}
By an integral test we obtain the upper bound (recalling formula 3.251.2 from~\cite{GR81})
\begin{align*}
\frac{1}{2y'}\sum\limits_{n=1}^{\infty}\frac{1}{\big(1+\big(\frac{n}{2y'}\big)^{2}\big)^{k+\varepsilon}}
\leq\int\limits_{0}^{\infty}\frac{1}{(1+\nu^{2})^{k+\varepsilon}}\,\mathrm{d}\nu=\frac{\sqrt{\pi}\,
\Gamma(k-1/2+\varepsilon)}{2\,\Gamma(k+\varepsilon)}\,.
\end{align*}
Using now that $Y\leq y'\leq k/(2\pi)$, we arrive at the upper bound
\begin{align*}
\sum\limits_{\gamma\in\Gamma_{p_{j}}\setminus\{\mathrm{id}\}}\sigma(z,\gamma z)^{-(k+\varepsilon)}
\leq\frac{k\,\Gamma(k-1/2+\varepsilon)}{\sqrt{\pi}\,\Gamma(k+\varepsilon)}\,.
\end{align*}
An application of an effective version of Stirling's formula (see Lemma~\ref{lemma-6.3} of the Appendix) 
gives 
\begin{align*}
\frac{k\,\Gamma(k-1/2+\varepsilon)}{\sqrt{\pi}\,\Gamma(k+\varepsilon)}\leq\frac{k\,e^{5/4}}{\sqrt{\pi}
\,\sqrt{k+\varepsilon}}\,,
\end{align*}
which completes the proof of the lemma.
\end{proof}

\section{Main results}
\label{section-5}
Based on the upper bounds for the Poincar\'e series $P_{k,\varepsilon}^{\Gamma}(z)$ established 
for $z$ ranging through the compact domain $\mathcal{F}_{Y}$ in Subsection~\ref{subsection-3.2} 
and for $z$ ranging through the cuspidal neigborhoods $\mathcal{F}^{Y}_{j}$ in Subsection~\ref
{subsection-4.2}, we are now in position to state and prove the main results of this paper providing 
upper bounds for the supremum of the quantity $S_{2k}^{\Gamma}(z)$ in the cocompact as well 
as in the cofinite setting. We also address the question of lower bounds for the quantity under 
consideration. We end this section with some explicit computations in the case of the modular 
group $\Gamma=\mathrm{PSL}_{2}(\mathbb{Z})$.

\subsection{Main results in the cocompact setting}
\label{subsection-5.1}
In this subsection, we will give an effective upper bound for the supremum of the quantity $S_{2k}^
{\Gamma}(z)$ for $k\in\mathbb{N}_{\geq 2}$, when $z$ is ranging through the compact domain 
$\mathcal{F}_{Y}$. In particular, this will lead us to effective upper and lower bounds for the 
supremum of the quantity $S_{2k}^{\Gamma}(z)$, when $\Gamma$ is cocompact and torsionfree. 
We start by establishing an upper bound for the quantity $S_{2k}^{\Gamma}(z)$ in terms of the 
Poincar\'e series $P_{k,\varepsilon}^{\Gamma}(z)$, which is valid for all $z\in\mathbb{H}$.

\begin{prop}
\label{proposition-5.1}
Let $k\in\mathbb{N}_{\geq 1}$ and $0<\varepsilon<1$. Then, the inequality
\begin{align*}
S_{2k}^{\Gamma}(z)\leq\frac{(2k-1+\varepsilon)(1+\varepsilon)}{4\pi}+\frac{3(2k+\varepsilon)(2k-
1+\varepsilon)(1+\varepsilon)}{4\pi(k+\varepsilon)}\sum\limits_{\gamma\in\Gamma\setminus
\{\mathrm{id}\}}\sigma(z,\gamma z)^{-(k+\varepsilon)}
\end{align*}
holds for $z\in\mathbb{H}$.
\end{prop}
\begin{proof}
Letting $\lambda=s(1-s)$ and $\mu=t(1-t)$ with $s,t\in W_{k}\cap\mathbb{R}$ such that $t>s>1$, 
formula~\eqref{formula-2.3} of Lemma~\ref{lemma-2.1} states the equality
\begin{align*} 
&\sum\limits_{j=0}^{\infty}\bigg(\frac{1}{\lambda_{j}-\lambda}-\frac{1}{\lambda_{j}-\mu}\bigg)\vert
\varphi_{j}(z)\vert^{2}+\frac{1}{4\pi}\sum\limits_{j=1}^{h}\int\limits_{-\infty}^{\infty}\bigg(\frac{1}{\frac
{1}{4}+r^{2}-\lambda}-\frac{1}{\frac{1}{4}+r^{2}-\mu}\bigg)\bigg\vert E_{j}\bigg(z,\frac{1}{2}+ir\bigg)
\bigg\vert^{2}\,\mathrm{d}r \\
&= -\frac{1}{4\pi}\big(\psi(s+k)+\psi(s-k)-\psi(t+k)-\psi(t-k)\big)+\sum\limits_{\gamma\in\Gamma
\setminus\{\mathrm{id}\}}\bigg(\frac{c\bar{z}+d}{cz+d}\bigg)^{k}\bigg(\frac{z-\gamma\bar{z}}
{\gamma z-\bar{z}}\bigg)^{k}g_{k}(s;z,\gamma z).
\end{align*}
Restricting the summation on the left-hand side of the above formula to the eigenvalue $\lambda_
{j}=k(1-k)$ and neglecting all the other summands and taking absolute values on the right-hand 
side, then yields the inequality
\begin{align*}
&\sum\limits_{j=1}^{d_{2k}}\bigg(\frac{1}{k(1-k)-s(1-s)}-\frac{1}{k(1-k)-t(1-t)}\bigg)\vert f_{j}(z)
\vert^{2}y^{2k} \\
&\qquad\leq\frac{1}{4\pi}\big\vert\psi(s+k)+\psi(s-k)-\psi(t+k)-\psi(t-k)\big\vert+\sum\limits_
{\gamma\in\Gamma\setminus\{\mathrm{id}\}}\big\vert g_{k}(s;z,\gamma z)\big\vert.
\end{align*}
Next we choose $s=k+\varepsilon$ and $t=s+1=k+1+\varepsilon$, and compute
\begin{align}
\notag
r(k,\varepsilon)&:=\frac{1}{k(1-k)-s(1-s)}-\frac{1}{k(1-k)-t(1-t)} \\[1mm]
\notag
&\,\,=\frac{1}{\varepsilon(2k-1+\varepsilon)}-\frac{1}{2k+\varepsilon(2k+1+\varepsilon)} \\[1mm]
\notag
&\,\,=\frac{2(k+\varepsilon)}{\varepsilon(2k-1+\varepsilon)(2k+\varepsilon(2k+1+\varepsilon))} 
\\[1mm]
\label{formula-5.1}
&\,\,=\frac{2(k+\varepsilon)}{\varepsilon(2k+\varepsilon)(2k-1+\varepsilon)(1+\varepsilon)}\,.
\end{align}
Furthermore, recalling that the digamma function $\psi(s)$ satisfies the functional equation $\psi
(s+1)-\psi(s)=1/s$, leads to the relation
\begin{align}
\notag
&\psi(s+k)+\psi(s-k)-\psi(t+k)-\psi(t-k) \\[2mm]
\notag
&\qquad =\psi(2k+\varepsilon)+\psi(\varepsilon)-\psi(2k+1+\varepsilon)-\psi(1+\varepsilon) \\
\label{formula-5.2}
&\qquad =-\frac{1}{2k+\varepsilon}-\frac{1}{\varepsilon}=-\frac{2(k+\varepsilon)}{\varepsilon(2k+
\varepsilon)}\,.
\end{align}
Collecting the above calculations then gives the upper bound
\begin{align*}
\frac{2(k+\varepsilon)}{\varepsilon(2k+\varepsilon)(2k-1+\varepsilon)(1+\varepsilon)}\sum\limits_
{j=1}^{d_{2k}}\vert f_{j}(z)\vert^{2}y^{2k}\leq\frac{2(k+\varepsilon)}{4\pi\varepsilon(2k+\varepsilon)}
+\sum\limits_{\gamma\in\Gamma\setminus\{\mathrm{id}\}}g_{k}(k+\varepsilon;z,\gamma z),
\end{align*}
in other words, we have the upper bound
\begin{align*}
S_{2k}^{\Gamma}(z)\leq\frac{(2k-1+\varepsilon)(1+\varepsilon)}{4\pi}+\frac{\varepsilon(2k+
\varepsilon)(2k-1+\varepsilon)(1+\varepsilon)}{2(k+\varepsilon)}\sum\limits_{\gamma\in\Gamma
\setminus\{\mathrm{id}\}}g_{k}(k+\varepsilon;z,\gamma z),
\end{align*}
which is valid for all $z\in\mathbb{H}$. Since $0<\varepsilon<1$, Lemma~\ref{lemma-6.2} of the 
Appendix applies and provides for all $z\in\mathbb{H}$ and $\gamma\in\Gamma$ the inequality
\begin{align*}
g_{k}(k+\varepsilon;z,\gamma z)\leq\frac{3}{2\pi\varepsilon}\sigma(z,\gamma z)^{-(k+\varepsilon)},
\end{align*}
from which the claimed inequality follows.
\end{proof}

In the next theorem we prove the first part of Theorem~B given in the introduction.

\begin{thm}
\label{theorem-5.2} 
Let $k\in\mathbb{N}_{\geq 2}$ and $Y>0$. Then, with $\sigma_{Y}$ given in Definition~\ref
{definition-3.3} and $B_{Y}$ given in Definition~\ref{definition-3.5} , the upper bound
\begin{align*}
\sup_{z\in\mathcal{F}_{Y}}S_{2k}^{\Gamma}(z)&\leq\frac{2k-1}{4\pi}\bigg(1+6\sum\limits_{e_{j}
\in\mathcal{E}}(n_{j}-1)\bigg)+12(2k-1)B_{Y}\,\sigma_{Y}^{-(k-2)}
\end{align*}
holds.
\end{thm}
\begin{proof}
Given $0<\varepsilon<1$, Proposition~\ref{proposition-5.1} provides for all $z\in\mathbb{H}$ the 
upper bound
\begin{align}
\label{formula-5.3}
S_{2k}^{\Gamma}(z)\leq\frac{(2k-1+\varepsilon)(1+\varepsilon)}{4\pi}+\frac{3(2k+\varepsilon)(2k
-1+\varepsilon)(1+\varepsilon)}{4\pi(k+\varepsilon)}\sum\limits_{\gamma\in\Gamma\setminus
\{\mathrm{id}\}}\sigma(z,\gamma z)^{-(k+\varepsilon)}\,.
\end{align}
By means of Proposition~\ref{proposition-3.8}, we then obtain for $z\in\mathcal{F}_{Y}$ the upper
bound
\begin{align*}
S_{2k}^{\Gamma}(z)&\leq\frac{(2k-1+\varepsilon)(1+\varepsilon)}{4\pi}+\frac{3(2k+\varepsilon)(2k
-1+\varepsilon)(1+\varepsilon)}{4\pi(k+\varepsilon)}\times \\
&\hspace*{21mm}\times\bigg(4\pi\,\frac{2+\varepsilon}{1+\varepsilon}\,B_{Y}\,\sigma_{Y}^{-(k-2)}+
\sum\limits_{e_{j}\in\mathcal{E}}(n_{j}-1)\bigg).
\end{align*}
Letting $\varepsilon\rightarrow 0$, we thus arrive for $z\in\mathcal{F}_{Y}$ at the upper bound
\begin{align*}
S_{2k}^{\Gamma}(z)&\leq\frac{2k-1}{4\pi}\bigg(1+6\sum\limits_{e_{j}\in\mathcal{E}}(n_{j}-1)\bigg)
+12(2k-1)B_{Y}\,\sigma_{Y}^{-(k-2)}\,,
\end{align*}
which concludes the proof of the theorem.
\end{proof}

In the next theorem we prove Theorem~A given in the introduction.

\begin{thm} 
\label{theorem-5.3}
Let $\Gamma$ be cocompact and torsionfree, and let $k\in\mathbb{N}_{\geq 2}$. Then, the bounds
\begin{align*}
\frac{2k-1}{4\pi}\leq\sup_{z\in\mathbb{H}}S_{2k}^{\Gamma}(z)\leq\frac{2k-1}{4\pi}+C_{\Gamma}\,
e^{-\delta_{\Gamma}k}
\end{align*}
hold, where the constants $C_{\Gamma}$ and $\delta_{\Gamma}$ are effectively computable as
\begin{align*}
C_{\Gamma}=\frac{3\,e^{4\pi g_{\Gamma}/\ell_{\Gamma}}}{\pi(g_{\Gamma}-1)}\frac{(\cosh(\ell_
{\Gamma})+1)^{2}}{\log((\cosh(\ell_{\Gamma})+1)/2)}\quad\text{and}\quad\delta_{\Gamma}=\frac
{1}{2}\log\bigg(\frac{\cosh(\ell_{\Gamma})+1}{2}\bigg).
\end{align*}
\end{thm}
\begin{proof}
The lower bound has been proven in~\cite[Sec.~7.1]{FJK16}. As far as the proof of the upper bound 
is concerned, we recall that in the cocompact setting we have chosen $\mathcal{F}_{Y}=\mathcal{F}$, 
so that we obtain from Theorem~\ref{theorem-5.2} 
\begin{align*}
\sup_{z\in\mathbb{H}}S_{2k}^{\Gamma}(z)\leq\frac{2k-1}{4\pi}+12(2k-1)B_{Y}\,\sigma_{Y}^{-(k-2)},
\end{align*}
where we have from Lemma~\ref{lemma-3.4} and Definition~\ref{definition-3.5} that
\begin{align*}
\sigma_{Y}\geq\frac{\cosh(\ell_{\Gamma})+1}{2}\qquad\text{and}\qquad B_{Y}=\frac{e^{\mathrm
{diam}_{\mathrm{hyp}}(\mathcal{F})/2}}{\mathrm{vol}_{\mathrm{hyp}}(\mathcal{F})}\,,
\end{align*}
respectively; to simplify notations, we set $\sigma_{\Gamma}:=(\cosh(\ell_{\Gamma})+1)/2$. Using 
the inequality 
\begin{align*}
\mathrm{diam}_{\mathrm{hyp}}(\mathcal{F})\leq\frac{2\,\mathrm{vol}_{\mathrm{hyp}}(\mathcal{F})}
{\ell_{\Gamma}}\leq\frac{8\pi g_{\Gamma}}{\ell_{\Gamma}}
\end{align*}
proven in~\cite{Chang}, we can estimate $B_{Y}$ as
\begin{align*}
B_{Y}\leq\frac{e^{4\pi g_{\Gamma}/\ell_{\Gamma}}}{4\pi(g_{\Gamma}-1)}\,.
\end{align*}
Now, taking into account the inequalities
\begin{align*}
ax\leq e^{ax}\qquad\Longleftrightarrow\qquad xe^{-2ax}\leq\frac{e^{-ax}}{a}\,,
\end{align*}
which are valid for $a>0$ and $x\geq 0$, we derive by choosing $a=\log(\sigma_{\Gamma})$ and 
$x=k/2$ that
\begin{align*}
\frac{k}{2}\sigma_{\Gamma}^{-k}\leq\frac{e^{-\log(\sigma_{\Gamma})k/2}}{\log(\sigma_{\Gamma})}\,.
\end{align*}
Since $k\in\mathbb{N}_{\geq 2}$, we conclude from the above that
\begin{align*}
&12(2k-1)B_{Y}\,\sigma_{Y}^{-(k-2)}\leq 12(2k-1)B_{Y}\,\sigma_{\Gamma}^{-(k-2)} \\[1mm]
&\qquad\leq 48B_{Y}\,\sigma_{\Gamma}^{2}\,\frac{k}{2}\,\sigma_{\Gamma}^{-k}\leq\frac{12\,e^{4
\pi g_{\Gamma}/\ell_{\Gamma}}}{\pi(g_{\Gamma}-1)}\,\sigma_{\Gamma}^{2}\,\frac{e^{-\log(\sigma_
{\Gamma})k/2}}{\log(\sigma_{\Gamma})}\,,
\end{align*}
which proves the claim with the constants $C_{\Gamma}$ and $\delta_{\Gamma}$ as stated in the 
theorem.
\end{proof}

\subsection{Main results in the cofinite setting}
\label{subsection-5.2}
In this subsection, we will give an effective upper bound for the supremum of the quantity $S_{2k}^
{\Gamma}(z)$ for $k\in\mathbb{N}_{\geq 2}$, when $z$ is ranging through the cuspidal neighborhoods 
$\mathcal{F}^{Y}_{j}$. By Lemma~\ref{lemma-4.1}, we may assume that $Y<k/(2\pi)$, since in the 
case $Y\geq k/(2\pi)$ the desired upper bound is also provided by Theorem~\ref{theorem-5.2}. This
will prove the second part of Theorem~B given in the introduction.

\begin{thm}
\label{theorem-5.4} 
Let $k\in\mathbb{N}_{\geq 2}$, $Y_{0}>0$, and $Y:=\max\{2Y_{0},16/\sqrt{15}\}$; assume $Y< k/
(2\pi)$. Then, with $B_{k,Y_{0}}$ given in Definition~\ref{definition-4.3}, the upper bounds
\begin{align*}
\sup_{z\in\mathcal{F}^{Y}_{j}}S_{2k}^{\Gamma}(z)\leq\frac{2k-1}{4\pi}+\frac{3(2k-1)}{2\pi}\bigg(B_
{k,Y_{0}}+\frac{\sqrt{k}\,e^{5/4}}{\sqrt{\pi}}\bigg)=O(k^{3/2})
\end{align*}
hold for $j=1,\ldots,h$.
\end{thm}
\begin{proof}
Since the inequality $Y<k/(2\pi)$ holds by assumption, the second part of Lemma~\ref{lemma-4.1} 
allows us to restrict the range for $z$ from $\mathcal{F}^{Y}_{j}$ to $\mathrm{cl}(\mathcal{F}^{Y}_
{j}\setminus\mathcal{F}^{k/(2\pi)}_{j})$ in the subsequent estimates.

Given $0<\varepsilon<1$, Proposition~\ref{proposition-5.1} provides for all $z\in\mathbb{H}$ the 
upper bound
\begin{align*}
S_{2k}^{\Gamma}(z)\leq\frac{(2k-1+\varepsilon)(1+\varepsilon)}{4\pi}+\frac{3(2k+\varepsilon)(2k
-1+\varepsilon)(1+\varepsilon)}{4\pi(k+\varepsilon)}\sum\limits_{\gamma\in\Gamma\setminus
\{\mathrm{id}\}}\sigma(z,\gamma z)^{-(k+\varepsilon)}\,.
\end{align*}
By means of the decomposition
\begin{align*}
\sum\limits_{\gamma\in\Gamma\setminus\{\mathrm{id}\}}\sigma(z,\gamma z)^{-(k+\varepsilon)}
\leq\sum\limits_{\gamma\in\Gamma\setminus\Gamma_{p_{j}}}\sigma(z,\gamma z)^{-(k+
\varepsilon)}+\sum\limits_{\gamma\in\Gamma_{p_{j}}\setminus\{\mathrm{id}\}}\sigma(z,\gamma 
z)^{-(k+\varepsilon)}\,,
\end{align*}
we then obtain for $z\in\mathrm{cl}(\mathcal{F}^{Y}_{j}\setminus\mathcal{F}^{k/(2\pi)}_{j})$, using 
Lemma~\ref{lemma-4.4} and Lemma~\ref{lemma-4.5}, that
\begin{align*}
\sum\limits_{\gamma\in\Gamma\setminus\{\mathrm{id}\}}\sigma(z,\gamma z)^{-(k+\varepsilon)}
\leq B_{k,Y_{0}}(\varepsilon)+\frac{k\,e^{5/4}}{\sqrt{\pi}\,\sqrt{k+\varepsilon}}\,,
\end{align*}
which yields the upper bound
\begin{align*}
S_{2k}^{\Gamma}(z)\leq\frac{(2k-1+\varepsilon)(1+\varepsilon)}{4\pi}+\frac{3(2k+\varepsilon)(2k-
1+\varepsilon)(1+\varepsilon)}{4\pi(k+\varepsilon)}\bigg(B_{k,Y_{0}}(\varepsilon)+\frac{k\,e^{5/4}}
{\sqrt{\pi}\,\sqrt{k+\varepsilon}}\bigg).
\end{align*}
The proof of the theorem now follows by letting $\varepsilon\rightarrow 0$.
\end{proof}

The next proposition addresses the case $k=1$.

\begin{prop}
\label{proposition-5.5}
Let $0<\varepsilon<1$ and $Y\geq1/(2\pi)$. Then, with $B_{Y}$ given in Definition~\ref{definition-3.5}, 
the upper bound
\begin{align*}
\sup_{z\in\mathbb{H}}S_{2}^{\Gamma}(z)\leq\frac{(1+\varepsilon)^{2}}{4\pi}+\frac{3(1+\varepsilon)^
{2}(2+\varepsilon)}{\varepsilon}B_{Y}
\end{align*}
holds.
\end{prop}
\begin{proof}
From inequality~\eqref{formula-5.3}, which is easily verified to hold also for $k=1$, we obtain the 
upper bound
\begin{align*}
S_{2}^{\Gamma}(z)\leq\frac{(1+\varepsilon)^{2}}{4\pi}+\frac{3(1+\varepsilon)(2+\varepsilon)}{4\pi}
\sum\limits_{\gamma\in\Gamma\setminus\{\mathrm{id}\}}\sigma(z,\gamma z)^{-(1+\varepsilon)}
\end{align*}
for $z\in\mathbb{H}$. By means of Lemma~\ref{lemma-3.7}, we then arrive at the upper bound
\begin{align*}
S_{2}^{\Gamma}(z)\leq\frac{(1+\varepsilon)^{2}}{4\pi}+\frac{3(1+\varepsilon)^{2}(2+\varepsilon)}
{\varepsilon}B_{Y}
\end{align*}
for $z\in\mathcal{F}_{Y}$. Furthermore, since we have by assumption that $Y\geq 1/(2\pi)$, 
Lemma~\ref{lemma-4.1}, which is also valid for $k=1$, shows that the same upper bound is 
valid for $z\in\mathcal{F}^{Y}_{j}$ for $j=1,\ldots,h$. This proves the claim.
\end{proof}

\subsection{Lower bounds for the sup-norm of $S_{2k}^{\Gamma}(z)$}
\label{subsection-5.3}
In this subsection, we prove lower bounds for the supremum of the quantity $S_{2k}^{\Gamma}(z)$
for $z$ ranging through the compact domain $\mathcal{F}_{Y}$, as well as for $z$ ranging through
the cuspidal neighborhoods~$\mathcal{F}^{Y}_{j}$.

\begin{prop}
\label{proposition-5.6}
Let $g_{\Gamma}\geq 1$, $k\in\mathbb{N}_{\geq 1}$, and $Y\geq k/(2\pi)$. Then, the lower
bound
\begin{align*}
\sup_{z\in\mathcal{F}_{Y}}S_{2k}^{\Gamma}(z)\geq\frac{k-1}{2\pi}
\end{align*}
holds.
\end{prop}
\begin{proof}
We start from the obvious inequality
\begin{align*}
\sup_{z\in\mathcal{F}}S_{2k}^{\Gamma}(z)\cdot\mathrm{vol_{\mathrm{hyp}}}(M)\geq\int\limits_
{\mathcal{F}}S_{2k}^{\Gamma}(z)\mu_{\mathrm{hyp}}(z)=d_{2k},
\end{align*}
where we recall that
\begin{align*}
\mathrm{vol}_{\mathrm{hyp}}(M)&=2\pi\bigg((2g_{\Gamma}-2)+h+\sum\limits_{e_{j}\in\mathcal
{E}}\bigg(1-\frac{1}{n_{j}}\bigg)\bigg), \\
d_{2k}&=(2k-1)(g_{\Gamma}-1)+(k-1)h+\sum\limits_{e_{j}\in\mathcal{E}}\bigg\lfloor{k\bigg(1-\frac
{1}{n_{j}}\bigg)}\bigg\rfloor.
\end{align*}
Since $g_{\Gamma}\geq 1$, we arrive at the lower bound 
\begin{align*}
d_{2k}&\geq (2k-2)(g_{\Gamma}-1)+(k-1)h+\sum\limits_{e_{j}\in\mathcal{E}}\bigg(k\bigg(1-\frac
{1}{n_{j}}\bigg)-\bigg(1-\frac{1}{n_{j}}\bigg)\bigg) \\
&=(k-1)\bigg((2g_{\Gamma}-2)+h+\sum\limits_{e_{j}\in\mathcal{E}}\bigg(1-\frac{1}{n_{j}}\bigg)
\bigg).
\end{align*}
From this we derive
\begin{align*}
\sup_{z\in\mathcal{F}}S_{2k}^{\Gamma}(z)\geq\frac{d_{2k}}{\mathrm{vol}_{\mathrm{hyp}}(M)}
\geq\frac{k-1}{2\pi}\,.
\end{align*}
Since $Y\geq k/(2\pi)$, Lemma~\ref{lemma-4.1} then shows that
\begin{align*}
\sup_{z\in\mathcal{F}_{Y}}S_{2k}^{\Gamma}(z)=\sup_{z\in\mathcal{F}}S_{2k}^{\Gamma}(z)\geq
\frac{k-1}{2\pi}\,,
\end{align*}
which concludes the proof of the proposition.
\end{proof}

\begin{prop}
\label{proposition-5.7}
Let $k\in\mathbb{N}_{\geq 2}$, $\delta>0$, $Y_{0}>0$, and $Y:=\max\{2Y_{0},16/\sqrt{{15}}\}$.
Then, for $k\gg Y$, the lower bounds
\begin{align*}
\sup_{z\in\mathcal{F}^{Y}_{j}}S_{2k}^{\Gamma}(z)=\Omega(k^{3/2-\delta})
\end{align*}
hold for $j=1,\ldots,h$.
\end{prop}
\begin{proof}
Again, we work from formula~\eqref{formula-2.3} of Lemma~\ref{lemma-2.1} with $\lambda=
s(1-s)$ and $\mu=t(1-t)$ with $s,t\in W_{k}\cap\mathbb{R}$ such that $t>s>1$, which reads
\begin{align*} 
&\sum\limits_{j=0}^{\infty}\bigg(\frac{1}{\lambda_{j}-\lambda}-\frac{1}{\lambda_{j}-\mu}\bigg)\vert
\varphi_{j}(z)\vert^{2}+\frac{1}{4\pi}\sum\limits_{j=1}^{h}\int\limits_{-\infty}^{\infty}\bigg(\frac{1}{\frac
{1}{4}+r^{2}-\lambda}-\frac{1}{\frac{1}{4}+r^{2}-\mu}\bigg)\bigg\vert E_{j}\bigg(z,\frac{1}{2}+ir\bigg)
\bigg\vert^{2}\,\mathrm{d}r \\
&= -\frac{1}{4\pi}\big(\psi(s+k)+\psi(s-k)-\psi(t+k)-\psi(t-k)\big)+\sum\limits_{\gamma\in\Gamma
\setminus\{\mathrm{id}\}}\bigg(\frac{c\bar{z}+d}{cz+d}\bigg)^{k}\bigg(\frac{z-\gamma\bar{z}}
{\gamma z-\bar{z}}\bigg)^{k}g_{k}(s;z,\gamma z).
\end{align*}
Choosing $t=s+1$ and recalling that the smallest eigenvalue among the $\lambda_{j}$'s equals 
$k(1-k)$, we find that the left-hand side of the above formula as a function of $s$ has a simple 
pole of order $1$ at $s=k$ arising from the summands corresponding to the eigenvalue $k(1-k)$.
Therefore, letting $s=k+\varepsilon$ with $\varepsilon>0$ and $\lambda_{j}=k(1-k)$, we obtain 
after dividing both sides of the above formula by the quantity $r(k,\varepsilon)$ given by~\eqref
{formula-5.1}, for each cusp $p_{j}$ ($j=1,\ldots,h$) the equality
\begin{align*} 
S_{2k}^{\Gamma}(z)=&-\lim_{\varepsilon\rightarrow 0}\frac{1}{4\pi r(k,\varepsilon)}\big(\psi(2k+
\varepsilon)+\psi(\varepsilon)-\psi(2k+1+\varepsilon)-\psi(\varepsilon+1)\big) \\
&+\lim_{\varepsilon\rightarrow 0}\frac{1}{r(k,\varepsilon)}\sum\limits_{\gamma\in\Gamma \setminus 
\Gamma_{p_{j}}}\bigg(\frac{c\bar{z}+d}{cz+d}\bigg)^{k}\bigg(\frac{z-\gamma\bar{z}}{\gamma z-\bar
{z}}\bigg)^{k}g_{k}(k+\varepsilon;z,\gamma z) \\ 
&+\lim_{\varepsilon\rightarrow 0}\frac{1}{r(k,\varepsilon)}\sum\limits_{\gamma\in\Gamma_{p_{j}} 
\setminus\{\mathrm{id}\}}\bigg(\frac{c\bar{z}+d}{cz+d}\bigg)^{k}\bigg(\frac{z-\gamma\bar{z}}{\gamma 
z-\bar{z}}\bigg)^{k}g_{k}(k+\varepsilon;z,\gamma z).
\end{align*}
Formulas~\eqref{formula-5.1} and~\eqref{formula-5.2} show that the first summand on the right-hand 
side of the above formula is of order $O(k)$. Furthermore, since we have assumed that $k\gg Y$, we
can suppose that $Y<k/(2\pi)$, and Lemma~\ref{lemma-4.4} together with Lemma~\ref{lemma-4.1} 
shows that for $z\in\mathcal{F}^{Y}_{j}$, the second summand of the above formula is also of order 
$O(k)$. We are thus left to prove that the third summand is of order $\Omega(k^{3/2-\delta})$ for $k
\gg Y$. To this end, we let $z=\sigma_{j}z'$ with the scaling matrix $\sigma_{j}$ of the cups $p_{j}$, 
and compute
\begin{align*}
&\lim_{\varepsilon\rightarrow 0}\frac{1}{r(k,\varepsilon)}\sum_{\gamma\in\Gamma_{p_{j}}\setminus
\{\mathrm{id}\}}\bigg(\frac{c\bar{z}+d}{cz+d}\bigg)^{k}\bigg(\frac{z-\gamma\bar{z}}{\gamma z-\bar{z}}
\bigg)^{k}g_{k}(k+\varepsilon;\sigma(z,\gamma z)) \\
&\qquad=\lim_{\varepsilon\rightarrow 0}\frac{1}{r(k,\varepsilon)}\sum\limits_{\gamma'\in\sigma_{j}^
{-1}\Gamma_{p_{j}}\sigma_{j}\setminus\{\mathrm{id}\}}\bigg(\frac{c'\bar{z}'+d'}{c'z'+d'}\bigg)^{k}\bigg
(\frac{z'-\gamma'\bar{z}'}{\gamma' z'-\bar{z}'}\bigg)^{k}g_{k}(k+\varepsilon;\sigma(z',\gamma' z')) \\
&\qquad=\lim_{\varepsilon\rightarrow 0}\frac{1}{r(k,\varepsilon)}\sum\limits_{\substack{n\in\mathbb
{Z}\\n\neq 0}}\bigg(\frac{z'-\bar{z}'-n}{z'-\bar{z}'+n}\bigg)^{k}g_{k}(k+\varepsilon;\sigma(z',z'+n)).
\end{align*}
Now we note that the latter quantity is independent of the specific Fuchsian subgroup $\Gamma$. 
However, it has been shown in~\cite[Sec.~7.2]{FJK16} for the modular group $\Gamma=\mathrm
{PSL}_{2}(\mathbb{Z})$ that the latter quantity is of order $\Omega(k^{3/2-\delta})$ for $k\gg Y$. 
This completes the proof of the proposition.
\end{proof}

\subsection{Explicit computations for the modular group $\Gamma=\mathrm{PSL}_{2}(\mathbb
{Z})$}
\label{subsection-5.4}
In this subsection, we illustrate how Theorem~\ref{theorem-5.2} and Theorem~\ref{theorem-5.4} 
lead to effective upper bounds for the supremum of the quantity $S^{\Gamma}_{2k}(z)$ in the 
case of the modular group $\Gamma=\mathrm{PSL}_{2}(\mathbb{Z})$. The proof of this result 
gives rise to an algorithm to determine effective upper bounds for the supremum of the quantity 
$S_{2k}^{\Gamma}(z)$ for more general Fuchsian subgroups $\Gamma$; this algorithm is 
reproduced in Subsection~\ref{subsection-6.3} of the Appendix.

\begin{thm} 
\label{theorem-5.8}
Let $\Gamma=\mathrm{PSL}_{2}(\mathbb{Z})$, $k\in\mathbb{N}$, and $Y=16/\sqrt{15}=4.132...$ 
Then, the upper bounds
\begin{align*}
S_{2k}^{\Gamma}(z)\leq
\begin{cases}
\displaystyle
\frac{31(2k-1)}{4\pi}+72(2k-1)1.014^{-(k-2)}\qquad\qquad\qquad\quad\,\,\,\text{if }k\geq 2,\,z
\in\mathcal{F}_{Y}, \\[2mm]
\displaystyle
\frac{31(2k-1)}{4\pi}+72(2k-1)1.014^{-(k-2)}\qquad\qquad\quad\,\text{if }2\leq k\leq 25,\,z\in
\mathcal{F}^{Y}_{1}, \\[2mm]
\displaystyle
\frac{2k-1}{4\pi}+\frac{3(2k-1)}{2\pi}\bigg(4^{-k+4}\bigg(\frac{k}{2\pi}\bigg)^{4}+\frac{\sqrt{k}\,e^
{5/4}}{\sqrt{\pi}}\bigg)\qquad\text{if }k\geq 26,\,z\in\mathcal{F}^{Y}_{1},
\end{cases}
\end{align*}
hold.\end{thm}
\begin{proof}
For the subsequent proof, the reader will have to recall various notations that have been introduced 
in Section~\ref{section-2}.

(1) We start by choosing the standard fundamental domain for the quotient space \linebreak 
$\mathrm{PSL}_{2}(\mathbb{Z})\backslash\mathbb{H}$, which is given as
\begin{align*}
\mathcal{F}=\{z=x+iy\in\mathbb{C}\,\vert\,\vert z\vert\geq 1,\,-1/2\leq x\leq +1/2\}.
\end{align*}
Its boundary $\partial\mathcal{F}$ consists of the set of geodesic line segments $\mathcal{S}=
\{S_{1},S_{2},S_{3},S_{4}\}$, where
\begin{align*}
&S_{1}:=\{z=-1/2+iy\,\vert\,y\geq\sqrt{3}/2\},\quad S_{2}:=\{z=x+iy\,\vert\,\vert z\vert=1,\,-1/2\leq 
x\leq 0\}, \\[1mm]
&S_{3}:=\{z=+1/2+iy\,\vert\,y\geq\sqrt{3}/2\},\quad S_{4}:=\{z=x+iy\,\vert\,\vert z\vert=1,\,0\leq x
\leq +1/2\}.
\end{align*}
Since the minimal positive trace of a hyperbolic element $\gamma\in\mathrm{PSL}_{2}(\mathbb
{Z})$ must be at least $3$ and since $\gamma=\big(\begin{smallmatrix}2&1\\1&1\end{smallmatrix}
\big)$ is an element having this trace, the length $\ell_{\Gamma}$ of the shortest closed geodesic 
on $\mathrm{PSL}_{2}(\mathbb{Z})\backslash\mathbb{H}$ is easily computed to
\begin{align*}
\ell_{\Gamma}=2\cosh^{-1}(3/2)=1.924...
\end{align*}
(2) The set of cusps of $\mathcal{F}$ is given by $\mathcal{P}=\{p_{1}\}$, where $p_{1}:=i\infty$; 
for the corresponding scaling matrix we have $\sigma_{1}=\mathrm{id}$. The set of elliptic fixed
points of $\mathcal{F}$ is given by $\mathcal{E}=\{e_{1}, e_{2},e_{3}\}$, where
\begin{align*}
e_{1}:=\frac{-1+i\sqrt{3}}{2},\quad e_{2}:=i,\quad e_{3}:=\frac{1+i\sqrt{3}}{2}\,;
\end{align*}
from this we immediately get that $\theta_{\Gamma}=2\pi/3$.

(3) We now choose $Y_{0}:=2$; from Theorem~\ref{theorem-5.4} we then find that $Y=4.131...$
From this and the above choices, we get $m_{Y}=\sqrt{3}/2$ and $M_{Y}=4.131...$

(4) In this step we determine the quantity $\mu_{\Gamma}$ given by~\eqref{formula-3.1}. With the 
notations of steps (1) and (2), we first need to calculate the hyperbolic distances $\mathrm{dist}_
{\mathrm{hyp}}(S_{j},e_{h})$ for $j\in\{1,2,3,4\}$ and $h\in\{1,2,3\}$ subject to the condition that 
$e_{h}\notin S_{j}$. By symmetry, it suffices to consider the following three cases:
\begin{align*}
\mathrm{dist}_{\mathrm{hyp}}(S_{1},e_{2}),\quad\mathrm{dist}_{\mathrm{hyp}}(S_{3},e_{1}),\quad
\mathrm{dist}_{\mathrm{hyp}}(S_{4},e_{1}).
\end{align*}
In the first case we compute
\begin{align*}
\cosh\big(\mathrm{dist}_{\mathrm{hyp}}(S_{1},e_{2})\big)&=\min_{y\geq\sqrt{3}/2}\bigg(1+\frac
{\vert -1/2+iy-i\vert^{2}}{2y}\bigg) \\
&=\min_{y\geq\sqrt{3}/2}\bigg(\frac{y}{2}+\frac{5}{8y}\bigg)=\frac{\sqrt{5}}{2},
\end{align*}
which gives
\begin{align*}
\mathrm{dist}_{\mathrm{hyp}}(S_{1},e_{2})=0.481...
\end{align*}
In a similar way, we find in the remaining two cases
\begin{align*}
&\mathrm{dist}_{\mathrm{hyp}}(S_{3},e_{1})=\cosh^{-1}\bigg(\frac{\sqrt{7}}{\sqrt{3}}\bigg)=0.986...
\,, \\
&\mathrm{dist}_{\mathrm{hyp}}(S_{4},e_{1})=\mathrm{dist}_{\mathrm{hyp}}(e_{2},e_{1})=\cosh^
{-1}\bigg(\frac{2\sqrt{3}}{3}\bigg)=0.549...
\end{align*}
Using~\eqref{formula-3.1} we arrive at 
\begin{align*}
\mu_{\Gamma}=\min\{0.481...\,,0.986...\,,0.549...\}=0.481...
\end{align*}
(5) In this step we estimate the quantity $\sigma_{Y}$ given by~\eqref{formula-3.2} using Lemma
\ref{lemma-3.4}. With the results of steps (1)--(4), we find
\begin{align*}
\sigma_{Y}&\geq\min\bigg\{\frac{\cosh(\ell_{\Gamma})+1}{2},\,\sinh^{2}(\mu_{\Gamma})\,\sin^{2}
\bigg(\frac{\theta_{\Gamma}}{2}\bigg)+1,\,\frac{m_{Y}^{2}}{4}+1,\,\frac{1}{4M_{Y}^{2}}+1\bigg\} 
\\[1mm]
&=\min\{2.248...,1.187...,1.187...,1.014...\}\geq 1.014.
\end{align*}
(6) In this step we give crude upper bounds for the hyperbolic diameters of $\mathcal{F}_{Y}$ and 
$\mathcal{F}_{Y_{0}}$. In order to estimate $\mathrm{diam}_{\mathrm{hyp}}(\mathcal{F}_{Y})$, 
we consider the rectangle $\mathcal{R}\subset\mathbb{H}$ with vertices 
\begin{align*}
\{-1/2+ia,+1/2+ia,+1/2+ib,-1/2+ib\},
\end{align*}
where $a=\sqrt{3}/{2}$ and $b=Y$. For $z,w\in\mathcal{R}$, we then have the bounds
\begin{align*}
\vert z-w\vert^{2}\leq\vert(-1/2+ia)-(+1/2+ib)\vert^{2}=1+(b-a)^{2}\quad\text{and}\quad\mathrm
{Im}(z)\mathrm{Im}(w)\geq a^{2}. 
\end{align*}
Using the formula
\begin{align*}
\cosh\big(\mathrm{dist}_{\mathrm{hyp}}(z,w)\big)=1+\frac{\vert z-w\vert^{2}}{2\mathrm{Im}(z)
\mathrm{Im}(w)}\,,
\end{align*}
we find the upper bound
\begin{align*}
\mathrm{dist}_{\mathrm{hyp}}(z,w)\leq\cosh^{-1}\bigg(1+\frac{1+(b-a)^{2}}{2a^{2}}\bigg),
\end{align*}
and thus can estimate the hyperbolic diameter of $\mathcal{F}_{Y}$ as
\begin{align*}
\mathrm{diam}_{\mathrm{hyp}}(\mathcal{F}_{Y})\leq 2.861...
\end{align*}
In a similar way, we find for the hyperbolic diameter of $\mathcal{F}_{Y_{0}}$ the upper bound
\begin{align*}
\mathrm{diam}_{\mathrm{hyp}}(\mathcal{F}_{Y_{0}})\leq 1.577...
\end{align*}
(7) In this step we compute the hyperbolic volumes of $\mathcal{F}_{Y}$ and $\mathcal{F}_{Y_
{0}}$. For the hyperbolic volume of $\mathcal{F}_{Y}$, we obtain
\begin{align*}
\mathrm{vol}_{\mathrm{hyp}}(\mathcal{F}_{Y})=\int\limits_{-1/2}^{1/2}\int\limits_{\sqrt{1-x^{2}}}^
{16/\sqrt{15}}\frac{\mathrm{d}y\wedge\mathrm{d}x}{y^{2}}=\int\limits_{-1/2}^{1/2}\bigg(\frac{1}
{\sqrt{1-x^{2}}}-\frac{\sqrt{15}}{16}\bigg)\mathrm{d}x=0.805...
\end{align*}
In a similar way, we find for the hyperbolic volume of $\mathcal{F}_{Y_{0}}$ the result
\begin{align*}
\mathrm{vol}_{\mathrm{hyp}}(\mathcal{F}_{Y_{0}})=0.547...
\end{align*}
(8) In this step we estimate the quantities $B_{Y}$ and $B_{Y_{0}}$ given by~\eqref{formula-3.3}. 
Applying the results obtained in steps (6) and (7), we get the upper bound
\begin{align*}
B_{Y}=\frac{e^{\mathrm{diam}_{\mathrm{hyp}}(\mathcal{F}_{Y})/2}}{\mathrm{vol}_{\mathrm{hyp}}
(\mathcal{F}_{Y})}\leq 5.193...
\end{align*}
In a similar way, we find $B_{Y_{0}}\leq 4.022...$ Recalling~\eqref{formula-4.2}, we derive 
from this
\begin{align*}
B_{k,Y_{0}}=4^{-k+3}\,2\pi\,Y_{0}^{-4}B_{Y_{0}}\bigg(\frac{k}{2\pi}\bigg)^{4}\leq 4^{-k+4}\bigg
(\frac{k}{2\pi}\bigg)^{4}.
\end{align*}
(9) For $k\geq 2$ and $z\in\mathcal{F}_{Y}$, or for $2\leq k\leq 2\pi Y=25.956...$ and $z\in
\mathcal{F}^{Y}_{1}$, we obtain from Theorem~\ref{theorem-5.2}, taking into account the 
bounds obtained in steps (5) and (8), that the upper bound
\begin{align*}
S_{2k}^{\Gamma}(z)\leq\frac{31(2k-1)}{4\pi}+72(2k-1)1.014^{-(k-2)}
\end{align*}
holds, which proves the first two parts of the theorem.

(10) Finally, for $k>2\pi Y=25.956...$ and $z\in\mathcal{F}^{Y}_{1}$, we obtain from Theorem
\ref{theorem-5.4}, taking into account the bounds obtained in step (8), that the upper bound
\begin{align*}
S_{k}^{\Gamma}(z)\leq\frac{2k-1}{4\pi}+\frac{3(2k-1)}{2\pi}\bigg(4^{-k+4}\bigg(\frac{k}{2\pi}
\bigg)^{4}+\frac{\sqrt{k}\,e^{5/4}}{\sqrt{\pi}}\bigg)
\end{align*}
holds, which proves the last part of the theorem.
\end{proof}

\section{Appendix}
\label{section-6}
For the sake of completeness we collect in this appendix some basic facts about the resolvent
kernel and the heat kernel for the hyperbolic Laplacian $\Delta_{k}$. Furthermore, we provide 
an effective version of Stirling's formula and end the appendix with an algorithm formalizing the 
proof of Theorem~\ref{theorem-5.8}.

\subsection{The resolvent kernel}
\label{subsection-6.1}
In this subsection, we give the basic definitions of the resolvent kernel and the heat kernel for the 
hyperbolic Laplacian $\Delta_{k}$, as well as the representation of the resolvent kernel as an 
integral transform of the heat kernel. Furthermore, we provide an upper bound for the resolvent 
kernel which is crucial for the main results of this paper.

\subsubsection*{Definition of the resolvent kernel.}
\label{subsubsection-6.1.1}
Let $F(a,b;c;Z)$ be the hypergeometric series with variable $Z$ and parameters $a,b,c\in\mathbb
{C}$ such that $-c\in\mathbb{N}$ is allowed only if $-a\in\mathbb{N}$ and $a>c$, or if $-b\in\mathbb
{N}$ and $b>c$. For $Z\in\mathbb{C}$ with $\vert Z\vert<1$, the hypergeometric series then has 
the power series expansion (see~\cite[p.~37]{MAO})
\begin{align*}
F(a,b;c;Z)=\sum\limits_{n=0}^{\infty}\frac{(a)_{n}(b)_{n}}{(c)_{n}}\frac{Z^{n}}{n!} 
\end{align*}
with the Pochhammer symbols $(a)_{n}=\Gamma(a+n)/\Gamma(a)$ etc.. 

Following~\cite[\S~1.4]{Fischer} (see also~\cite{ElstrodtResolv}), the resolvent kernel $G_{k}(s;z,w)$ 
on $\mathbb{H}$ associated to $\Delta_{k}$ ($k\in\mathbb{N}_{\geq 1}$) is defined for $s\in W_{k}
=\mathbb{C}\setminus\{k-n,\,-k-n\,\vert\,n\in\mathbb{N}\}$ and $z,w\in\mathbb{H}$ by the formula
\begin{align*}
G_{k}(s;z,w):=G_{k}(s;\sigma(z,w)),
\end{align*}
where 
\begin{align*}
\sigma(z,w)=\cosh^{2}\bigg(\frac{\mathrm{dist}_{\mathrm{hyp}}(z,w)}{2}\bigg)
\end{align*}
is the displacement function~\eqref{formula-2.1} and
\begin{align*}
G_{k}(s;\sigma):=\frac{1}{\sigma^{s}}\frac{\Gamma(s+k)\Gamma(s-k)}{4\pi\Gamma(2s)}\,F\bigg(s+
k,s-k;2s;\frac{1}{\sigma}\bigg)\qquad(\sigma\geq 1).
\end{align*}

\subsubsection*{Definition of the heat kernel.}
\label{subsubsection-6.1.2}
Following~\cite{Oshima}, correcting a corresponding formula in~\cite{Fay}, the heat kernel $K_{k}(t;
z,w)$ on $\mathbb{H}$ associated to $\Delta_{k}$ ($k\in\mathbb{N}_{\geq 1}$) is defined for $t\in
\mathbb{R}_{\geq 0}$ and $z,w\in\mathbb{H}$ by the formula
\begin{align*}
K_{k}(t;z,w):=K_{k}(t;\mathrm{dist}_{\mathrm{hyp}}(z,w)),
\end{align*}
where
\begin{align*}
K_{k}(t;\rho):=\frac{\sqrt{2}\,e^{-t/4}}{(4\pi t)^{3/2}}\int\limits_{\rho}^{\infty}\frac{re^{-r^{2}/(4t)}}{\sqrt
{\cosh(r)-\cosh(\rho)}}\,T_{2k}\bigg(\frac{\cosh(r/2)}{\cosh(\rho/2)}\bigg)\,\mathrm{d}r\qquad(\rho
\geq 0),
\end{align*}
with
\begin{align*}
T_{2k}(X):=\cosh(2k\,\textrm{arccosh}(X))
\end{align*}
denoting the $2k$-th Chebyshev polynomial. In~\cite{FJK16}, we have shown that $K_{k}(t;\rho)$ 
is a monotone decreasing function of $\rho$ and that the inequality
\begin{align} 
\label{formula-6.1}
T_{2k}\big(\cosh(r/2)\big)\leq e^{kr}
\end{align}
holds. Using the upper bound~\eqref{formula-6.1}, we then derive for later purposes for $\rho\geq 
0$, the estimate
\begin{align} 
\notag
K_{k}(t;\rho)&\leq K_{k}(t;0)=\frac{\sqrt{2}\,e^{-t/4}}{(4\pi t)^{3/2}}\int\limits_{0}^{\infty}\frac{re^{-r^{2}/
(4t)}}{\sqrt{\cosh(r)-1}}\,T_{2k}\big(\cosh(r/2)\big)\,\mathrm{d}r \\
\notag
&=\frac{\sqrt{2}\,e^{-t/4}}{(4\pi t)^{3/2}}\int\limits_{0}^{\infty}\frac{re^{-r^{2}/(4t)}}{\sinh(r/2)}\,T_{2k}
\big(\cosh(r/2)\big)\,\mathrm{d}r\leq\frac{\sqrt{2}\,e^{-t/4}}{(4\pi t)^{3/2}}\int\limits_{0}^{\infty}\frac{r
e^{-r^{2}/(4t)}}{\sinh(r/2)}\,e^{kr}\,\mathrm{d}r \\
\label{formula-6.2}
&\leq\frac{C_{\delta}\,e^{-t/4}}{t^{3/2}}\int\limits_{0}^{\infty}e^{(\delta-1/2)r}e^{-r^{2}/(4t)}e^{kr}\,
\mathrm{d}r\leq\frac{C'_{\delta}\,e^{-t/4}}{t}\,e^{t(k-1/2+\delta)^{2}};
\end{align}
here $\delta>0$ is arbitrarily small and the positive constants $C_{\delta}$, $C'_{\delta}$ depend 
solely on $\delta$.

\subsubsection*{Resolvent kernel as an integral transform of the heat kernel.}
\label{subsubsection-6.1.3}
The resolvent kernel $G_{k}(s;z,w)$ on $\mathbb{H}$ associated to $\Delta_{k}$ can be represented 
as an integral transform of the heat kernel $K_{k}(t;z,w)$ on $\mathbb{H}$ associated to $\Delta_
{k}$; the precise relationship is given as
\begin{align}
\label{formula-6.3}
G_{k}(s;\sigma)=\int\limits_{0}^{\infty}e^{-(s-1/2)^{2}t}e^{t/4}K_{k}(t;\rho)\,\mathrm{d}t,
\end{align}
where $\sigma=\cosh^{2}(\rho/2)$. We note that by~\eqref{formula-6.2}, formula~\eqref{formula-6.3} 
is valid for $\mathrm{Re}(s)\geq k+\varepsilon$ for any $\varepsilon>0$. We emphasize that we will 
be able to obtain useful estimates for the resolvent kernel by viewing it as the integral transform~\eqref
{formula-6.3} of the heat kernel and applying some of the estimates that have been established in~\cite
{FJK16}.

Next, we recall the function $g_{k}(s;z,w)$, which has been defined for $s\in W_{k}$ and $z,w\in
\mathbb{H}$ by means of formula~\eqref{formula-2.2}; in the present notation this leads to
\begin{align*}
g_{k}(s;\sigma):=G_{k}(s;\sigma)-G_{k}(s+1;\sigma).
\end{align*}
Using~\eqref{formula-6.3}, the function $g_{k}(s;\sigma)$ can be rewritten as
\begin{align*}
g_{k}(s;\sigma)=\int\limits_{0}^{\infty}\big(e^{-(s-1/2)^{2}t}-e^{-(s+1/2)^{2}t}\big)e^{t/4}K_{k}(t;\rho)\,
\mathrm{d}t;
\end{align*}
again, we have $\sigma=\cosh^{2}(\rho/2)$. 

\subsubsection*{Estimates for the resolvent kernel.}
\label{subsubsection-6.1.4}
Letting $a,b\in\mathbb{R}$ with $b\neq 0$ and using the formula
\begin{align*}
\int\limits_{0}^{\infty}t^{-3/2}e^{-a^{2}t-b^{2}/(4t)}\,\mathrm{d}t=\frac{2\sqrt{\pi}e^{-ab}}{b}\,,
\end{align*}
we compute for $s\geq k+\varepsilon$ with $\varepsilon>0$
\begin{align}
\notag
g_{k}(s;\sigma)&=\int\limits_{0}^{\infty}\big(e^{-(s-1/2)^{2}t}-e^{-(s+1/2)^{2}t}\big)e^{t/4}K_{k}(t;\rho)
\,\mathrm{d}t \\
\notag
&=\int\limits_{0}^{\infty}\int\limits_{\rho}^{\infty}\frac{\sqrt{2}}{(4\pi t)^{3/2}}\big(e^{-(s-1/2)^{2}t}-e^
{-(s+1/2)^{2}t}\big)\frac{re^{-r^{2}/(4t)}}{\sqrt{\cosh(r)-\cosh(\rho)}}\,T_{2k}\bigg(\frac{\cosh(r/2)}
{\cosh(\rho/2)}\bigg)\,\mathrm{d}r\,\mathrm{d}t \\[2mm]
\label{formula-6.4}
&=\frac{1}{2\pi\sqrt{2}}\int\limits_{\rho}^{\infty}\frac{e^{-(s-1/2)r}-e^{-(s+1/2)r}}{\sqrt{\cosh(r)-\cosh
(\rho)}}\,T_{2k}\bigg(\frac{\cosh(r/2)}{\cosh(\rho/2)}\bigg)\,\mathrm{d}r.
\end{align}

To establish the crucial upper bound for the function $g_{k}(s;\sigma)$, we need the following lemma.

\begin{lem} 
\label{lemma-6.1} 
Let $k\in\mathbb{N}_{\geq 1}$ and $0<\varepsilon<1$. Then, for $s=k+\varepsilon$, the upper 
bound
\begin{align*}
\int\limits_{\rho}^{\infty}\frac{\big(e^{-(s-1/2)r}-e^{-(s+1/2)r}\big)\,e^{kr}}{\sqrt{\cosh(r)-\cosh(\rho)}}
\,\mathrm{d}r\leq\frac{3\sqrt{2}}{\varepsilon}e^{-\varepsilon\rho}
\end{align*}
holds.
\end{lem}
\begin{proof}
Since $s=k+\varepsilon$, we have to estimate the function
\begin{align*}
F(\rho):=\int\limits_{\rho}^{\infty}\frac{e^{-ar}-e^{-br}}{\sqrt{\cosh(r)-\cosh(\rho)}}\,\mathrm{d}r,
\end{align*}
where $a:=-1/2+\varepsilon$ and $b:=a+1=1/2+\varepsilon$. Using integration by parts, we then 
obtain
\begin{align*}
F(\rho)&=2 \int\limits_{\rho}^{\infty}\frac{e^{-ar}-e^{-br}}{\sinh(r)}\frac{\mathrm{d}}{\mathrm{d}r}\big
(\cosh(r)-\cosh(\rho)\big)^{1/2}\,\mathrm{d}r \\
&=-4\int\limits_{\rho}^{\infty}\big(\cosh(r)-\cosh(p)\big)^{1/2}\frac{\mathrm{d}}{\mathrm{d}r}\bigg
(\frac{e^{-ar}-e^{-br}}{e^{r}-e^{-r}}\bigg)\,\mathrm{d}r.
\end{align*}
With the above choices of $a$ and $b$, we compute
\begin{align*}
\frac{e^{-ar}-e^{-br}}{e^{r}-e^{-r}}=\frac{e^{-\varepsilon r}(e^{r/2}-e^{-r/2})}{e^{r}-e^{-r}}=\frac{e^
{-\varepsilon r}}{e^{r/2}+e^{-r/2}}\,.
\end{align*}
Hence, we get
\begin{align*} 
F(\rho)=-2\int\limits_{\rho}^{\infty}\big(\cosh(r)-\cosh(\rho)\big)^{1/2}\frac{\mathrm{d}}{\mathrm{d}r}
\bigg(\frac{e^{-\varepsilon r}}{\cosh(r/2)}\bigg)\,\mathrm{d}r.
\end{align*}
Observing that
\begin{align*}
\frac{\mathrm{d}}{\mathrm{d}r}\bigg(\frac{e^{-\varepsilon r}}{\cosh(r/2)}\bigg)\leq 0
\end{align*}
for $\rho\leq r<\infty$, and using the estimate
\begin{align*}
\big(\cosh(r)-\cosh(\rho)\big)^{1/2}\leq\big(\cosh(r)-1\big)^{1/2}=\sqrt{2}\sinh(r/2),
\end{align*}
we arrive at the upper bound
\begin{align*}
F(\rho)&\leq -2\sqrt{2}\int\limits_{\rho}^{\infty}\sinh(r/2)\frac{\mathrm{d}}{\mathrm{d}r}\bigg(\frac{e^
{-\varepsilon r}}{\cosh(r/2)}\bigg)\,\mathrm{d}r \\
&=2\sqrt{2}\,e^{-\varepsilon\rho}\tanh(\rho/2)+ \frac{\sqrt{2}\,e^{-\varepsilon\rho}}{\varepsilon};
\end{align*}
for the last equality, we used integration by parts once again. Since $0<\varepsilon<1$, we complete 
the proof of the lemma by employing the crude upper bound $\tanh(\rho/2)<1/\varepsilon$.
\end{proof}

\begin{lem}
\label{lemma-6.2}
Let $k\in\mathbb{N}_{\geq 1}$ and $0<\varepsilon<1$. Then, the upper bound
\begin{align*}
g_{k}(k+\varepsilon;\sigma)\leq\frac{3}{2\pi\varepsilon}\,\sigma^{-(k+\varepsilon)}
\end{align*}
holds.
\end{lem}
\begin{proof}
In~\cite{FJK16}, we have shown that
\begin{align*}
T_{2k}\bigg(\frac{\cosh(r/2)}{\cosh(\rho/2)}\bigg)=\cosh\bigg(2k\,\mathrm{arccosh}\bigg(\frac{\cosh
(r/2)}{\cosh(\rho/2)}\bigg)\bigg)\leq\frac{e^{kr}}{\cosh^{2k}(\rho/2)}.
\end{align*}
Hence, using formula~\eqref{formula-6.4}, the above estimate, and Lemma~\ref{lemma-6.1}, we 
obtain the upper bound
\begin{align*}
g_{k}(s,\sigma)&=\frac{1}{2\pi\sqrt{2}}\int\limits_{\rho}^{\infty}\frac{e^{-(s-1/2)r}-e^{-(s+1/2)r}}{\sqrt
{\cosh(r)-\cosh(\rho)}}\,T_{2k}\bigg(\frac{\cosh(r/2)}{\cosh(\rho/2)}\bigg)\,\mathrm{d}r \\ 
&\leq\frac{1}{2\pi\sqrt{2}}\int\limits_{\rho}^{\infty}\frac{e^{-(s-1/2)r}-e^{-(s+1/2)r}}{\sqrt{\cosh(r)-\cosh
(\rho)}}\frac{e^{kr}}{\cosh^{2k}(\rho/2)}\,\mathrm{d}r \\
&=\frac{1}{2\pi\sqrt{2}\cosh^{2k}(\rho/2)}\int\limits_{\rho}^{\infty}\frac{\big(e^{-(s-1/2)r}-e^{-(s+1/2)r}
\big)\,e^{kr}}{\sqrt{\cosh(r)-\cosh(\rho)}}\,\mathrm{d}r \\
&\leq\frac{3}{2\pi\varepsilon}\frac{e^{-\varepsilon\rho}}{\cosh^{2k}(\rho/2)}.
\end{align*}
Recalling that $\sigma=\cosh^{2}(\rho/2)\leq e^{\rho}$, we easily conclude the proof of the lemma.
\end{proof}

\subsection{Effective version of Stirling's formula}
\label{subsection-6.2}
In this subsection, we provide an effective version of Stirling's formula.

\begin{lem}
\label{lemma-6.3}
Let $Z\geq 1$. Then, the upper bound
\begin{align*}
\frac{\Gamma(Z-1/2)}{\Gamma(Z)}\leq\frac{e^{5/4}}{\sqrt{Z}}
\end{align*}
holds.
\end{lem}
\begin{proof}
We set
\begin{align*}
R(Z):=\log(\Gamma(Z))-\bigg(Z-\frac{1}{2}\bigg)\log(Z)+Z-\frac{1}{2}\log(2\pi).
\end{align*}
It follows that
\begin{align*}
&\log\bigg(\frac{\Gamma(Z-1/2)}{\Gamma(Z)}\bigg) \\[1mm]
&\quad=R\big(Z-\tfrac{1}{2}\big)-R(Z)+Z\big(\log\big(Z-\tfrac{1}{2}\big)-\log(Z)\big)-\log\big(Z-\tfrac
{1}{2}\big)+\tfrac{1}{2}\log(Z)+\tfrac{1}{2} \\[2mm]
&\quad=R\big(Z-\tfrac{1}{2}\big)-R(Z)+Z\big(\log\big(Z-\tfrac{1}{2}\big)-\log(Z)\big)+\log(Z)-\log\big
(Z-\tfrac{1}{2}\big)-\tfrac{1}{2}\log(Z)+\tfrac{1}{2}.
\end{align*}
Now, we estimate the last expression as follows: From~\cite[p.~257, eq.~(6.1.42)]{abram}, we recall 
the inequality $0<R(Z)\leq 1/(12Z)$, which gives for $Z\geq 1$
\begin{align*}
R\bigg(Z-\frac{1}{2}\bigg)-R(Z)\leq\bigg\vert R\bigg(Z-\frac{1}{2}\bigg)-R(Z)\bigg\vert\leq\frac{1}{12
(Z-1/2)}+\frac{1}{12Z}\leq\frac{3}{12}.
\end{align*}
Next, using the power series expansion of the logarithm, we get the estimate
\begin{align*}
\log\bigg(Z-\frac{1}{2}\bigg)-\log(Z)=\log\bigg(1-\frac{1}{2Z}\bigg)\leq -\frac{1}{2Z}.
\end{align*}
Finally, using that
\begin{align*}
\log(Z)-\log\bigg(Z-\frac{1}{2}\bigg)=-\log\bigg(1-\frac{1}{2Z}\bigg)\leq\log(2)<1,
\end{align*}
we find the upper bound
\begin{align*}
\log\bigg(\frac{\Gamma(Z-1/2)}{\Gamma(Z)}\bigg)\leq\frac{1}{4}-\frac{Z}{2Z}+1-\frac{1}{2}\log(Z)+
\frac{1}{2}=\frac{5}{4}-\frac{1}{2}\log(Z),
\end{align*}
which concludes the proof of the lemma.
\end{proof}

\subsection{The algorithm}
\label{subsection-6.3}
In this subsection, following the proof of Theorem~\ref{theorem-5.8}, we reproduce an algorithm that 
determines effective sup-norm bounds for $S_{2k}^{\Gamma}(z)$ for general Fuchsian subgroups 
$\Gamma$.
\texttt{
\begin{itemize}
\item[(1)]
Determine a closed and connected fundamental domain $\mathcal{F}$ of $\Gamma$.
\item[(2)]
Determine the set $\mathcal{S}$ of geodesic line segments forming $\partial\mathcal{F}$.
\item[(3)]
Determine the length $\ell_{\Gamma}$ of the shortest closed geodesic on $\Gamma\backslash
\mathbb{H}$.
\item[(4)]
Determine the set $\mathcal{P}$ of cusps of $\mathcal{F}$ and their scaling matrices.
\item[(5)]
Determine the set $\mathcal{E}$ of elliptic fixed points in $\mathcal{F}$ and their orders.
\item[(6)]
Choose $Y_{0}>0$, set $Y=\max\{2Y_{0},16/\sqrt{15}\}$, and determine $m_{Y}$ and $M_{Y}$.
\item[(7)]
Determine an upper bound for the quantity $\mu_{\Gamma}$ given by~\eqref{formula-3.1}.
\item[(8)]
Determine a lower bound for the quantity $\sigma_{Y}$ given by~\eqref{formula-3.2}.
\item[(9)]
Determine the hyperbolic diameters of $\mathcal{F}_{Y}$ and $\mathcal{F}_{Y_{0}}$.
\item[(10)]
Determine the hyperbolic volumes of $\mathcal{F}_{Y}$ and $\mathcal{F}_{Y_{0}}$.
\item[(11)]
Determine upper bounds for the quantities $B_{Y}$ and $B_{Y_{0}}$ given by~\eqref{formula-3.3}, \\
as well as for the quantity $B_{k,Y_{0}}$ given by~\eqref{formula-4.2}.
\item[(12)]
For $k\geq 2$ and $z\in\mathcal{F}_{Y}$, or for $2\leq k<2\pi Y$ and $z\in\mathcal{F}^{Y}_{j}$, 
determine an \\ upper bound for $S_{2k}^{\Gamma}(z)$ using Theorem~\ref{theorem-5.2}.
\item[(13)]
For $k\geq 2\pi Y$ and $z\in\mathcal{F}^{Y}_{j}$, determine an upper bound for $S_{2k}^
{\Gamma}(z)$ using \\ Theorem~\ref{theorem-5.4}.
\end{itemize}
}

\newpage

\noindent
Joshua S. Friedman \\
Department of Mathematics and Science \\
\textsc{United States Merchant Marine Academy} \\
300 Steamboat Road \\
Kings Point, NY 11024 \\
U.S.A. \\
e-mail: FriedmanJ@usmma.edu

\vspace{2mm}

\noindent
Jay Jorgenson \\
Department of Mathematics \\
The City College of New York \\
Convent Avenue at 138th Street \\
New York, NY 10031 \\
U.S.A. \\
e-mail: jjorgenson@mindspring.com

\vspace{2mm}

\noindent
J\"urg Kramer \\
Institut f\"ur Mathematik \\
Humboldt-Universit\"at zu Berlin \\
Unter den Linden 6 \\
D-10099 Berlin \\
Germany \\
e-mail: kramer@math.hu-berlin.de

\end{document}